\documentclass[11pt,a4paper]{article}
\usepackage{indentfirst}
\usepackage{amsmath}
\usepackage{amsfonts}
\usepackage{amsthm}
\usepackage{amssymb}
\usepackage{graphics}
\usepackage{graphicx}
\usepackage{multicol}

\def\H{\mathcal{H}}
\def\EE{\mathcal{EE}}

\def\N{\mathbb{N}}

\def\P{\mathbb{P}}
\def\p{\mathbf{p}}
\def\E{\mathbb{E}}

\def\EE{\mathcal{E}}

\def\R{\mathbb{R}}

\def\t{\textrm}
\def\d{\textrm{d}}

\def\w{\widetilde}

\def\ind{{\mathchoice {\rm 1\mskip-4mu l} {\rm 1\mskip-4mu l}
{\rm 1\mskip-4.5mu l} {\rm 1\mskip-5mu l}}}

\newcommand{\be} {\begin{equation}}
\newcommand{\ee} {\end{equation}}
\newcommand{\bea} {\begin{eqnarray}}
\newcommand{\eea} {\end{eqnarray}}
\newcommand{\Bea} {\begin{eqnarray*}}
\newcommand{\Eea} {\end{eqnarray*}}

\newtheorem{Thm}{Theorem}
\newtheorem{Lem}[Thm]{Lemma}
\newtheorem{Pte}[Thm]{Proposition}
\newtheorem{Prop}[Thm]{Proposition}

\theoremstyle{definition} 
\theoremstyle{definition} 
\theoremstyle{definition} 

\theoremstyle{remark}\newtheorem{Rque}{Remark}
\begin{document}
\title{Large deviations for Branching Processes \\ in Random Environment}
\author{Vincent Bansaye, Julien Berestycki}
\maketitle \vspace{1.5cm}

\begin{abstract}
\noindent A branching process in random environment $(Z_n ,n
\in  \N)$ is a generalization of Galton Watson processes where at
each generation the reproduction law is picked randomly. In this
paper we give several results which belong to the class of {\it
large deviations}. By contrast to the Galton-Watson case, here
random environments and the branching process can conspire to
achieve atypical events such as $Z_n \le e^{cn}$ when $c$ is
smaller than the typical geometric growth rate $\bar L$ and $ Z_n
\ge e^{cn}$ when $c > \bar L$.

One way to obtain such an atypical rate of growth is to have a
typical realization of the branching process in an atypical
sequence of environments. This gives us a general lower bound for
the rate of decrease of their probability.

When each individual leaves at least one offspring in the next
generation almost surely, we compute the exact rate function of
these events and we show that conditionally on the large deviation
event, the trajectory $t \mapsto \frac1n \log Z_{[nt]}, t\in [0,1]$
converges to a deterministic function $f_c :[0,1] \mapsto \R_+$ in
probability in the sense of the uniform norm. The most interesting
case is when $c < \bar L$ and we authorize individuals to have only
one offspring in the next generation. In this situation,
conditionally on $Z_n \le e^{cn}$, the population size stays fixed
at 1 until a time  $ \sim n t_c$. After time $n t_c$ an atypical
sequence of environments let $Z_n$ grow with the appropriate rate ($
\neq \bar L$) to reach $c.$ The corresponding map $f_c(t)$ is
piecewise linear and is 0 on $[0,t_c]$ and $f_c(t) =
c(t-t_c)/(1-t_c)$ on $[t_c,1].$
\end{abstract}



\clearpage

\section{Introduction}

Let $\mathcal P$ be the space of probability measures on the integer, that is
$$
\mathcal P  := \{ p : \N \mapsto [0,1] : \sum_{k\geq 0} p(k)=1 \},
$$
and denote by $m(p)$ the mean of $p$ :
$$m(p)=\sum_{k\geq 0} k p(k).$$

A branching process in random environment (BPRE for short) $(Z_n, n
\in \N)$ with environment distribution $\mu \in
\mathcal{M}_1(\mathcal \P)$ is a discrete time Markov process which
evolves as follows : at time $n,$  we draw $\p$ according to $\mu$
independently of the past and then each individual $i=1, \ldots ,
Z_n$ reproduces independently according to the same $\p,$ i.e. the
probability that individual $i$ gives birth to $k$ offsprings in the
next generation is $\p(k)$ for each $i$. We will denote by
$\P_{z_0}$ the distribution probability of this process started from
$z_0$ individuals. When we write $\P$ and unless otherwise
mentioned, we mean that the initial state is equal to $1$.

\medskip

Thus, we consider an i.i.d. sequence of  random environment
$(\p_i)_{i\in\N}$ with common distribution $\mu$. Traditionally, the
study of BPRE has relied on analytical tools such as generating
functions. More precisely, denoting by $f_i$ the probability
generating function of $\p_i$, one can note that the BPRE $(Z_n, n
\in \N)$ is characterized by the relation
$$\E\big(s^{Z_{n+1}}\vert Z_0,\dots,Z_n; \ f_0,\dots,f_n\big)=f_n(s)^{Z_n} \qquad (0\leq s\leq 1, \ n\geq 0).$$
For classical references on these processes see \cite{afa, at, atr,
AN, bpree, bpre}.

\medskip

A good picture to keep in mind when thinking of a BPRE is the
following : consider a population of plants which have a one year
life-cycle (so generations are discrete and non-overlapping). Each
year the climate or weather conditions (the environment) vary  which
impacts the reproductive success of the plant. Given the climate,
all the plants reproduce according to the same given mechanism. In
this context, $\mu$ can be thought of as the distribution which
controls the successive climates, which are supposed to be iid, and
the plant population then obeys a branching process in random
environment. By taking a Dirac mass for $\mu$ we recover the
classical case of Galton Watson processes.

\medskip

At least intuitively one easily sees that some information on the
behavior of the BPRE $Z_n$ can be read from the process $M_n=
\Pi_1^n m(\p_i)$  and that their typical behavior should be similar
: $$Z_n\approx M_n, \qquad (n\in\N).$$ Hence the following dichotomy
is hardly surprising: A BPRE is supercritical (resp. critical, resp.
subcritical) if the expectation of  $\log(m(\p))$ with respect to
$\mu$ the law of the environments : $$\E(\log(m(\p))),$$ is positive
(resp. zero, resp. negative). In the supercrticial case, the BPRE
survives with a positive  probability, in the critical and
subcritical case, it becomes extinct a.s.

Moreover, in the supercritical case, we have the following
expected result \cite{atr, Guiv}. Assuming that  $\E(\sum_{k\in\N}
k^s\p(k)/m(\p))<\infty$ for some $s>1$, there exists a finite
r.v. $W$ such that
$$M_n^{-1} Z_n \stackrel{n\rightarrow \infty}{\longrightarrow } W, \qquad \P(W>0)=\P(\forall n, Z_n>0).$$
which ensures that conditionally on the non-extinction of $(Z_n)_{n\in\N}$
$$\log(Z_n)/n \rightarrow  \E(\log(m(\p))) \quad \t{a.s}.$$
This result is  a generalization in random environment of the well
known Kesten-Stigum Theorem for Galton- Watson processes : let $N$
be the reproduction law of the GW process $(Z_n, n \ge 0)$ and let $m=\E(N)$ be its mean.
Assume that $E (N \log_+ N)<\infty$, then
$$W_n:=Z_n/(m^n) \stackrel{n\to \infty}{\longrightarrow } W, \qquad \P(W>0)=\P(\forall n, Z_n>0).$$
The distribution of $W$ is completely determined by that of $N$ and
a natural question concerns the tail behavior of $W$ near 0 and
infinity. Results in this direction can be found for instance in
\cite{BB, Dubuc, Dubuc2, Rouault} for the Galton Watson case and
\cite{Hambly} for the BPRE case. In a large deviation context, the
tail behavior of $W$ can be related to event where $Z_n$ grows with
an atypical rate. Another way to study such events is to consider
the asymptotic behavior of $Z_{n+1}/Z_n.$ This is the approach taken
in \cite{atld} to prove that $\vert W_n-W \vert$ decays
supergeometrically when $n \to \infty$, assuming that $\P(N=0)=0$.
Yet another approach is the study of so-called moderate deviations
(see \cite{Ney} for the asymptotic behavior of $\P(Z_n=v_n)$ with
$v_n=O(m^n)$).

Finally, we observe that Kesten Stigum Theorem for Galton Watson
processes can be reinforced into the following statement:
$$ (t \mapsto \frac1n \ln Z_{[nt]} , t\in[0,1] ) \Rightarrow (t \mapsto t \log (m), t\in [0,1]).$$
in the sense of the uniform norm almost surely (see for instance
\cite{Morters} for this type of trajectorial results, unconditioned
and conditioned on abnormally low growth rates).

\vspace{1cm}

In this work we will consider large deviation events for BPREs
$A_c(n), c \ge 0$ of the form
$$A_c(n) = \left\{ \begin{array}{c} \{0< \frac1n \log Z_n \le c \} \text{ for } c< \E(\log(m(\p)) \\ \{ \frac1n \log Z_n \ge c \} \text{ for } c>
\E(\log(m(\p)) \end{array} \right. ,$$ and we are interested in how
fast the probability of such events is decaying. More precisely, we
are interested in the cases where
$$-\frac1n \log (\P(A_c(n))) \to \chi(c), \qquad \t{with} \  \ \chi(c)<\infty.$$

Let us discuss very briefly the Galton-Watson case first (see
\cite{Fleischmann, Morters, Rouault}). Assume first that the Galton Watson process
is supercritical ($m:=\E(N) > 1$) and and that all the moments of
the reproduction law are finite. If we are in the B\"{o}ttcher case
($\P(N \le 1)=0$) then there are no large deviations, i.e.
$$c\ne \log m \ \Rightarrow \ \phi(c)=\infty.$$ If, on the other
hand, we are in the Schr\"{o}dder case ($\P(N=1)>0$) then $\phi(c)$
can be non-trivial for $c \le \log m.$ This case is discussed in
\cite{Morters} (see also \cite{Fleischmann} for finer results for
lower deviations) where it is shown that to achieve a
lower-than-normal rate of growth $c \le \log m$ the process first
refrains from branching for a long time until it can start to
grow at the normal rate $\log m$ and reach its objective. More
precisely, it is a consequence of Theorem \ref{T:trajectoire
gauche} below that conditionally on $Z_n \le e^{cn}$, $$ (\frac1n
\log (Z_{[nt]}) , t\in [0,1]) \to (f(t) , t\in[0,1])$$ in
probability in the sense of uniform norm, where $f(t) =  \log(m).
(t - (1- c/\log(m)))_+.$ When the reproduction law has infinite
moments, the rate function $\phi$ is non-trivial for $c \ge \log
m.$ In the critical or subcritical case, there are no large
deviations.

\medskip

We will see that the situation for BPRE differs in many aspects
from that of the Galton-Watson case: for instance the rate
function is non-trivial as soon as  $m(\p)$ is not constant and
more than $1$ with positive probability. This is due to the fact
that we can deviate following an atypical sequence of
environments, as explained in the next Section, and as already
observed by Kozlov for upper values in the supercritical case
\cite{kozld}. When we condition by $Z_n \le e^{cn}$ and we assume
$\P(Z_1 =1)>0$ the process $(\frac1n \log (Z_{[nt]}) , t\in
[0,1])$ still converges in probability uniformly to a function
$f_c(t)$ which has the same shape as $f$ above, that is there
exists $t_c  \in [0,1]$ such that $f_c(t)=0$ for $t \le t_c$ and
then $f_c$ is linear and reach $c$, but the slope of this later
piece can now differs from the typical rate $ \E(\log m(\p))$.

%In the supercritical case, when each individual leaves at least one
%descendent in the next generation $$\P(\p(0)=0)=1,$$ we are able to
%give an explicit expression of the rate function $\phi$ and to
%describe the trajectory of the BPRE $(Z_n)_{n\in\N}$ when it
%deviates. Informally, we prove in Section 2.1 that for lower
%deviation, $(Z_{i})_{0\leq i\leq n}$ stays bounded until an optimal
%time $t_c n$ and then $(\log(Z_i)/n)_{i/n \in [t_c,1]}$  follows a
%straight line to $c$. The precise value of $t_c$ minimizes the
%probabilistic cost of the strategy which consists in keeping the
%population size fixed at 1 until $n t_c$ and then getting an
%atypical sequence of climates which allows the population to grow
%geometrically with the desired rate. In Section 2.2., we prove that
%for upper deviation and under appropriate moment conditions, the
%process $(\log(Z_{i})/n)_{0\leq i\leq n}$ always follows a
%straight-line to $c>\E(\log(m(\p))$.

\section{Main results}
\label{S:main results}

Denote by  $(L_i)_{i\in\N}$ the sequence of
iid log-means of the successive environments,
$$L_i:=\log(m(\p_i)), \qquad S_n:=\sum_{i=0}^{n-1} L_i,$$
and
$$ \bar L:= \E(\log (m(\p))) = \E(L).$$

Define $\phi_L(\lambda):=\log (\E(\exp(\lambda L)))$ the Laplace transform of $L$
and let $\psi$ be the large deviation function associated with
$(S_n)_{n\in\N}$:
$$\psi(c)=\sup_{\lambda \in \R}\{ c \lambda-\phi_L(\lambda) \}.$$
We briefly recall some well known fact about the rate function $\psi$ (see \cite{dembo zeitouni} for a classical reference on the matter). The map $x \mapsto \psi(x)$ is strictly convex and $C^{\infty}$ in the interior of the set $\{ \Lambda'(\lambda), \lambda \in \mathcal{D}^o_{\Lambda} \}$  where $\mathcal{D}_{\Lambda} = \{\lambda : \Lambda(\lambda) <\infty\}.$ Furthermore, $\psi(\bar L)= 0$, and $\psi$ is decreasing (strictly) on the left of $\bar L$ and increasing (strictly) on its right.

The map $\psi$ is called the rate  function for the following
large deviation principle associated with the random walk $S_n$.
We have for every $c\leq  \bar{L}$, \be \label{ldS1}
  \lim_{n\rightarrow \infty}-
\log(\P(  S_n/n  \leq c ) /n =  \psi(c),
\ee
and for every $c\geq \bar{L}$
\be
\label{ldS2}
\lim_{n\rightarrow \infty} -
\log(\P( S_n/n \geq c  ) /n  =
 \psi(c).
\ee
$\newline$

Roughly speaking,  one way to get
$$ \log(Z_n)/n \in O \qquad (n\rightarrow \infty)$$
is to follow  environments with a good sequence of reproduction law :
$$\log(\Pi_{i=1}^n m(\p_i))/n=S_n/n \in O.$$
We have then the following upper bound for the rate function for
any BPRE under a moment condition analogue to that used in
\cite{Guiv}. The proof is deferred to the next section.

\begin{Prop} \label{P:minoration}  Assuming that  $\E(\sum_{k\in\N} k^s\p(k)/m(\p))<\infty$ for some $s>1$, then for
every $z_0$ :
\begin{itemize}
\item[-] $\forall c \le \bar L$ $$  \limsup_{n\rightarrow \infty}
- \frac1n \log \P_{z_0} ( \log(Z_n)/n \le c )  \leq \ \psi(c).$$
\item[-] $\forall c \ge \bar L$ $$  \limsup_{n\rightarrow \infty}
- \frac1n \log \P_{z_0} ( \log(Z_n)/n \ge c )  \leq \ \psi(c).$$
\end{itemize}
\end{Prop}

As Theorem \ref{T:trajectoire gauche} below shows, the inequality may be strict. Moreover, this proves that even in the subcritical case, there may be large deviations, contrary to what happens in the Galton Watson case. More precisely, as soon as $\P(m(\p)>1)>0$ and $m(\p)$ is not constant almost surely, the rate function $\psi$  is non trivial on $(0,\infty)$.

\subsection{Lower deviation in the strongly supercritical case.}
\label{lowerLD}
We focus here on the so-called {\it strongly supercritical} case
$$\P(\p(0)=0)=1$$
(in which the environments are almost surely supercritical). Let us  define for every $c \le \bar L$, \Bea \chi  (c) &:=& \inf_{
t \in [0,1]} \{ - t \log (\E(\p (1))) + (1-t) \psi(c/(1-t))\}. \Eea
It is quite easy to prove that this infimum is reached at a unique
point $t_c$ (see Lemma \ref{L:def de chi}):
$$\chi(c) = -t_c \log (\E(\p (1)))+ (1-t_c) \psi(c/(1-t_c)).$$
and that $t_c \in[0,1-c/\bar{L}]$. We can thus define the function
$f_c : [0,1] \mapsto \R_+$ for each $c < \bar L$ as follows (see
figure $1$):
%\begin{flushleft}
%\begin{multicols}{2}{
%\setlength\arraycolsep{-0pt}
$$f_c(t):= \left\{\begin{array}{ll}  0, \qquad &  \t{if} \ t\leq t_c \\
 \frac{c}{1-t_c}(t-t_c), \quad & \t{if} \ t\geq t_c.
\end{array}
\right.
$$
%$\newline$

\begin{figure}[h]\label{F:fc}
\begin{center}
\input{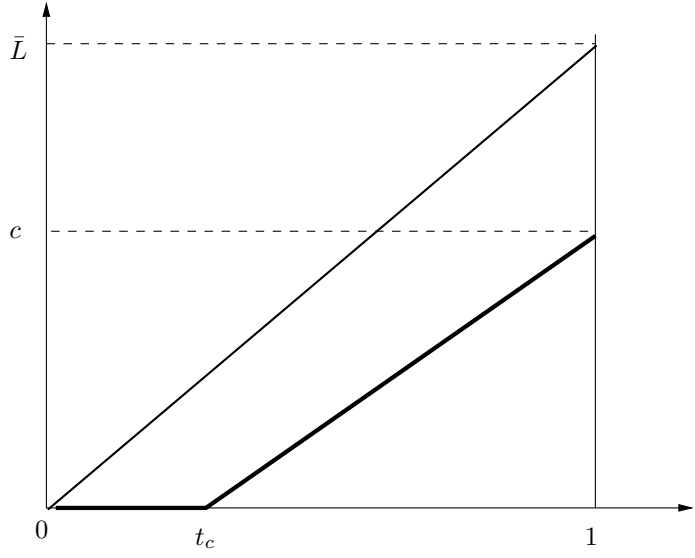} \caption{ The function $t \mapsto f_c(t)$ for
$c \le \bar L.$ }
\end{center}
\end{figure}

%}
%\includegraphics[scale=0.22]{image2.png}}
%\end{multicols}
%\end{flushleft}

We will need the following moment assumption $\H$.
\begin{align} \tag{$\H$}
\left\{  \begin{array}{c} \exists A>0  \t{ s.t.  } \mu(m(\p)>A) =0 , \\ \exists B>0  \t{ s.t.  } \mu(\sum_{k\in\N}
k^2\p(k)>B) =0 \end{array}  \right\}
\end{align}
Observe that the condition in Proposition \ref{P:minoration} ($\exists s>1$ such that $\E(\sum_{k\in\N} k^s\p(k)/m(\p))<\infty$) is included in $(\H)$. \\

The main result is the following theorem which gives the large deviation cost of $Z_n \le \exp(cn)$ and the
asymptotic trajectory behavior of $Z_n$ when conditioned on $Z_n \le \exp(cn).$ \\

\begin{Thm}\label{T:trajectoire gauche}
\begin{itemize} Assuming that $\P(\p(0) =0)=1$ and the hypothesis $\H$ we have
 \item[(a)] If  $\mu(\p(1)>0)>0$, then for every $c < \bar{L}$,
$$- \log (\P(Z_n \leq e^{ c n }))/n   \ \stackrel{n\rightarrow \infty}{\longrightarrow }  \chi(c),$$
 and furthermore, conditionally on   $Z_n \leq e^{ c n }$,
$$\sup_{t\in[0,1]} \{ \big\vert \log(Z_{[tn]})/n- f_c(t) \big\vert  \} \ \stackrel{n\rightarrow \infty}{\longrightarrow }0, \qquad \t{in} \ \P.$$

\item[(b)] If   $\mu(\p(1)>0)=0$,  then for every $c < \bar{L}$,
$$- \log (\P(Z_n <e^{ c n }))/n   \ \stackrel{n\rightarrow \infty}{\longrightarrow } \psi(c),$$
and furthermore  for every  $ \inf\{supp  \log(m(\p))\}<c<\bar L $,  conditionally on   $Z_n \leq e^{ c n }$,
$$\sup_{t\in[0,1]} \{ \big\vert \log(Z_{[tn]})/n- ct \big\vert   \}\
\stackrel{n\rightarrow \infty}{\longrightarrow }0, \qquad \t{in} \ \P.$$
\end{itemize}
\end{Thm}
$\newline$

Let us note that if  $\mu(\p(1)>0)>0$, then $t_c$ -the take-off point of the trajectory- may either be zero, either be equal to $1-c/\bar{L}$, or belong to $(0,1-c/\bar{L})$
(see Section \ref{S:example} for examples).

Moreover, when $m:=m(\p)$ is deterministic, as in the case of a GW process,
\begin{itemize}
\item[-] If $\mu(\p(1)>0)>0$ (B\"{o}ttcher case), then  $t_c=1-c/\log(m)$ and $\chi(c)=t_c \log(\E(\p(1)))$.
\item[-] If $\mu(\p(1)>0)=0$ (Schrodder case), then  $\chi(c)=-\infty$.
\end{itemize}
$\newline$

Let us first give a heuristic interpretation of the above
theorem. Observe that $$\P(Z_k=1, k=1,\ldots , t n) = \E(\p
(1))^{t n} = \exp( \log (\E(\p (1))) t n  ) $$ and that
\begin{align*}
\lim_{n \to \infty } \frac1n \log \P( S_{(1-t)n}/n \le c  ) &=
(1-t) \psi(c/(1-t) )
\end{align*}
so that we have $$\P( Z_k=1, k=1,\ldots , t n\ ; S_n - S_{t n}
\le c n   ) \asymp \exp( n [ t \log (\E(\p (1))) + (1-t)
\psi(c/(1-t))] ) $$ and $\chi(c)$ is just the ``optimal'' cost of
such an event with respect to the choice of $t.$  It is not hard
to see that the event $\{ Z_k=1, k=1,\ldots , t n\ ; S_n - S_{t n}
\le c n   \}$ is asymptotically included in $\{Z_n \leq  cn \}$ and
hence $\chi(c)$ is an upper bound for the rate function for
$Z_n$. Adding that once  $Z_n>>1$ is large enough it has no
choice but to follow the random walk of the log-means of the
environment sequence,  $\chi$ is actually the  good candidate to
be the rate function.
\\

Thus, roughly speaking, to deviate below $c,$ the process
$(\log(Z_{[nt]})/n)_{t\in [0,1]}$ stays bounded until an optimal
time $t_c$ and then deviates in straight line to $c$ thanks to a
good sequence of environments.
The proof in Section 5 and 6 follows this heuristic. \\

Another heuristic comment concerns the behavior of the environment
sequence conditionally on the event $Z_n \le e^{cn}.$ Before time
$[n t_c]$ we see a sequence of iid environments which are picked
according to the original probability law $\mu$ biased by $\p(1)$
the probability to have one offspring (think of the case where $\mu$
charges only two environments). After time $[n t_c ]$ we know that
the distribution of the sequence $(L_i)_{i \ge [nt_c]}$ is the law
of a sequence of iid $L_i$ conditioned on $\sum_{i = [n t_c]}^n L_i
\le [n c].$ This implies that the law of the environments is that of
an exchangeable sequence with common distribution $\mu$ tilted by
the log-means.

To conclude this section, we comment on the hypothesis
$\P(\p(0)=0)=1.$ It is known (see \cite{AN}) that for a Galton
Watson process $Z_n$ with survival probability $p$ and generating
function $f,$ under the $N \log N$ condition, for all $j \in \N$
\begin{align} \tag{*} \gamma^{-n} \P(Z_n =j) \to \alpha_j
\end{align}
 where $\forall
j \in \N : \alpha_j \in (0,\infty)$ and  $\gamma =f'(p).$ In the
case where $\P(Z_1=0)=0$ (no death), $\gamma = f'(p) =
f'(0)=\P(Z_1=1)$ which tells us that the cost of staying bounded is
the cost of keeping the population size fixed at $1$, a fact that we
also use for our analysis of BPRE. This suggests that the analogue
of $\gamma$ for BPRE should also play a role in the lower deviations
events when $\P(\p(0)=0)<1$. However there is not yet an analogue of
$(*)$ for BPRE and the situation is probably more complex.

\subsection{Upper deviation in the strongly supercritical case}
Assume as above that
$$\P(\p(0)=0)=1,$$
and that for every $k\geq 1$,
$$\E(Z_1^k)<\infty,$$
we have the following large deviation result for upper values.

\begin{Thm} \label{upperth}
For every  $c> \bar L$,
$$- \frac1n \log (\P(Z_n \geq e^{ c n }))  \stackrel{n\rightarrow \infty}{\longrightarrow } \psi(c),$$  and furthermore for $c< \sup\{supp  \log(m(\p))\}$, conditionally on $Z_n\geq \exp(cn)$,
$$\sup_{t\in[0,1]} \{ \big\vert \log(Z_{[tn]})/n- ct \big\vert  \} \ \stackrel{n\rightarrow \infty}{\longrightarrow }0.$$
\end{Thm}

To put it in words, this says that the cost of achieving a higher
than normal rate of growth is just the cost of seeing an atypical
sequence of environments in which this rate is expected.
Furthermore, conditionally  on $Z_n \ge e^{cn}$, the trajectory
$(\log(Z_{[nt]})/n)_{t\in [0,1]}$ is asymptotically a straight
line.

Kozlov \cite{kozld} gives the upper deviations of $Z_n$ in the case where
the generating functions $f$ are a.s. linear fractional and verify a.s. $f''(1)=2f'(1)^2$. In
the strongly supercritical case and under those hypothesis, he proves that for every $\theta>0$,  there exists $I(\theta)>0$ such that
$$\P(\log(Z_n)\geq  \theta n) \sim I(\theta) \P(S_n\geq  \theta n), \quad (n\rightarrow \infty).$$
Thus, Kozlov gets a finer result in the linear fractional case with  $f''(1)=2f'(1)^2$ a.s. by proving that the upper deviations of the BPRE $Z_n$  above $\bar L$ are exactly given by the large deviations of the random walk $S_n$.

Proposition \ref{P:minoration} shows that the rates of upper and lower deviations are at least those of the environments, but Theorem \ref{T:trajectoire gauche} and the remark below  show that the converse is not always true.

Theorem \ref{upperth} is the symmetric for upper deviations of case (b) of Theorem \ref{T:trajectoire gauche} for lower deviations. It is natural to ask if there is an analogue of case (a) as well. In this direction, we make the following two remarks.
\begin{itemize}
\item  If there exists $k>1$ such that
$$\E(Z_1^k)=\infty,$$
then the cost of reaching $c$ can be less that $\psi(c)$, since the BPRE might ``explode'' to a very large value in the first generation and then follow a geometric growth. This mirrors nicely what happens for lower deviations in the case (a). However we do not have an equivalent of Theorem \ref{T:trajectoire gauche} for upper deviations as such a result seems much harder to obtain for now.

\item In the case when
$$\P(m(\p)<1)>0,$$
then by Theorem 3 in \cite{Guiv},
$$s_{\max}:=\sup_{s\geq 1} \{\E(W^s)<\infty\}<\infty.$$
Thus, the BPRE $(Z_n)_{n\in\N}$ might deviate from the exponential of the random walk of
environments :
$$\lim_{n\rightarrow \infty} -\log(\P( \exp(-S_n) Z_n\geq \exp(n\epsilon))/n<\infty, \quad  (\epsilon>0),$$
which would yield a more complicated rate function for deviations.
\end{itemize}

\subsection{No large deviation without supercritical environment}

Finally, we consider the case when environments are a.s. subcritical or critical :
$$\P(m(\p)\leq 1)=1,$$
and we assume that for every $j\in \N$,
there exists $M_j>0$ such that
$$\sum_{k=0}^{\infty} k^j \p(k) \leq M_j \quad \t{a.s.} \qquad \ \ \ (\mathcal{M}).$$
Note that the condition $(\mathcal M)$ implies $(\H)$ simply by
considering $j=2.$ \\

In that case, even if $\P_1(Z_1\geq 2)>0$,  there is no large deviation, as in the case of a Galton Watson process.
\begin{Prop} \label{totalrecall} Suppose $(\mathcal{M})$ and that $\P(m(\p)\leq 1)=1$ , then for every $c>0$,
$$\lim_{n\rightarrow \infty} -\log(\P(Z_n\geq \exp(cn))/n=\infty.$$
\end{Prop}
We recall that by Proposition \ref{P:minoration}, this result does not hold if $\P(m(\p)> 1)>0$.
$\newline$

The next short section shows a concrete example where $t_c$ is non trivial. Section \ref{S:proof prop 1} is devoted
to the proof of Proposition \ref{P:minoration}. Section \ref{S:key lemmas}  is devoted to proving two key
lemmas which are then used repeatedly. The first gives the cost of keeping the population bounded for
a long time. The second tells us that once the population passes a threshold, it grows geometrically following the product of the means of environments. In Section \ref{S:proof}, we start by computing the rate function and then
 we describe the trajectory. Section \ref{S:proof upper}  is devoted to upper large deviation while
 Section \ref{S:subcritical} to case when environments are a.s. subcritical or critical.

\section{A motivating example : the case of two environments}
\label{S:example}

Suppose we have two environments $P_1$ and $P_2$ with $\mu(\p = P_1)= q.$ Call $L_1= \log m(P_1)$ and $L_2=\log m(P_2)$ their respective log mean and suppose $L_1 <L_2$. The random walk $S_n$ is thus the sum of iid variables $X : \P(X=L_1)=q, \P(X=L_2)=1-q.$

Recall that if $X$ is a Bernoulli variable with parameter $p$ the Fentchel Legendre transform of $\Lambda(\lambda) = \log (\E (e^{\lambda X}))$ is
$$
\Lambda^*(x) = x \log (x/p) + (1-x) \log ((1-x)/(1-p)).
$$
Hence the rate function for the large deviation principle associated to the random walk $S_n$ is
defined for  $L_1 \le x \le L_2$ by

$$
\psi(x)= z \log(z/p) +(1-z) \log ((1-z)/(1-p)) \text{ where } z = \frac{x-L_1}{L_2-L_1}.
$$

Recall that $\E( \p(1)) = qP_1(1) +(1-q)P_2(1)$ is the probability that an individual has exactly one descendent in the next generation. \\

The following figure 2 shows the function $t \mapsto -t \log (\E
(\p(1))) +(1-t)\psi(c/(1-t))$, so $\chi(c)$ is the minimum of this
function and $t_c$ is the $t$ where this minimum is reached. Figure
2 is drawn using the values $L_1=1, \ L_2=2, \ q=.5, \ \E (\p(1))
=.4$, $c=1.1$ and $1-c/\bar{L}\sim .27$. Thus, we ask $Z_n \leq
e^{1.1 n}$ whereas $Z_n$ behaves normally as $e^{1.5n}$ and this
example illustrate Theorem \ref{T:trajectoire gauche} a) with $t_c
\in (0,1-c/\bar{L}).$

\begin{figure}[h]\label{F:courbe}
 \begin{center}
\input{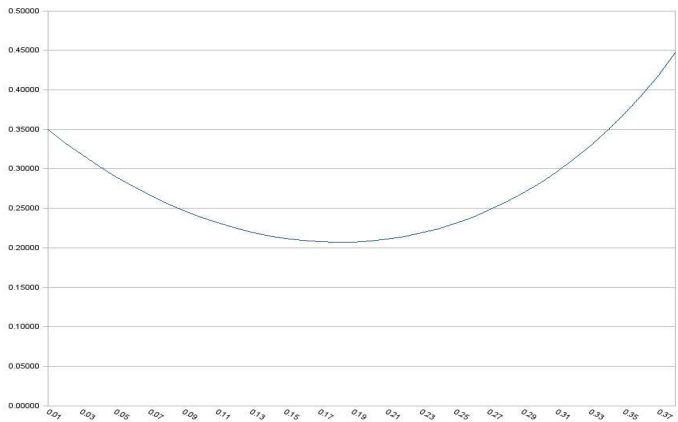}
\label{fig1}
\caption{In this example $t_c \sim 0.18$, the slope of the
function $f_c$ after $t_c$ is 1.34.  }
\end{center}
\end{figure}

$\newline$

As an illustration and a motivation we propose the following
model for parasites infection. In \cite{vb}, we consider a
branching model for parasite infection with cell division. In
every generation, the cells give birth to two daughter cells and
the cell population is the binary tree. We want to take into
account unequal sharing of parasites, following experiments made
in Tamara's Laboratory in Hopital Necker (Paris), and we
distinguish a first (resp. a second) daughter cell. Then we call
$Z^{(1)}$ (resp. $Z^{(2)}$) the number of descendants of a given
parasite of the mother cell in the first (resp. the second
daughter), where $(Z^{(1)},Z^{(2)})$ is any couple of random
variable (it may be non symmetric, dependent...). A key role for
limit theorems is played by the process $(Z_n)_{n\in\N}$ which
gives the number of parasites  in a random cell line (choosing
randomly one of the two daughter cells at each cell division and
counting the number of parasites inside). This process follows a
Branching process with two equiprobable environment with
respective reproduction law $Z^{(1)}$ and $Z^{(2)}$. Thus, here
$q=1/2$, $L_1=\log(\E(Z^{(1)}))$ and $L_2=\log(\E(Z^{(2)}))$.

We are interested in determining the number of cells with a large number of parasites and
we call $N_n^{\leq c}$ (resp $N_n^{\geq c}$) the number of cells in generation $n$ which
contain less (resp. more) than $\exp(cn)$ parasites, for $c>0$. An easy computation (which follows (17) in \cite{vb}) shows that
$$\E(N_n^{\leq c})=2^n\P(Z_n\leq \exp(cn)), \quad \E(N_n^{\geq c})=2^n\P(Z_n\geq \exp(cn)).$$
If $\P(Z^{(1)}=0)=\P(Z^{(2)}=0)=1$,  Section 2.1 ensures that for
every $c\geq \sqrt{\E(Z^{(1)})\E(Z^{(2)})}$,
$$\lim_{n\rightarrow \infty} \log(\E(N_n^{\leq c}))/n= \log(2)-\chi(c).$$
Moreover  Section 2.2 ensures that for every $c\geq \sqrt{\E(Z^{(1)})\E(Z^{(2)})}$,
$$\quad \lim_{n\rightarrow \infty} \log(\E(N_n^{\leq c}))/n= \log(2)-\psi(c).$$

\section{Proof of Proposition \ref{P:minoration} for general BPRE}
\label{S:proof prop 1}
Proposition \ref{P:minoration} comes from continuity of $\psi$ and
the following Lemma. \\
\begin{Lem}
For every $c>0$ and $z_0 \in \N$,
$$ \forall \epsilon>0, \quad \limsup_{n\rightarrow \infty}
-\frac1n\log(\P_{z_0}( c-\epsilon\leq \log(Z_n)/n \leq  c+\epsilon
)) \leq \psi(c).$$
\end{Lem}

\begin{proof}
Let $c>0$. Recall that $\phi_L(\lambda)=\E(\exp(\lambda L))$,
$$\psi(c)=\sup_{\lambda \in \R}\{\lambda c- \phi_L(\lambda)\},$$
and this supremum is reached in   $\lambda=\lambda_c$ such that
$$c=\phi'_L(\lambda_c)= \frac{\E(Le^{\lambda_c L})}{\E(e^{\lambda_c L})}=\frac{\E(m(\p)^{\lambda_c}\log(m(\p)))}{\E(m(\p)^{\lambda_c})}.$$

Then introduce the probability $\w{\P}$ on $\P$ defined by
$$\w{\P} ( \p  \in  \d p) =  \frac{ m(p)^{\lambda_c}}{\E( m(\p)^{\lambda_c})} \P(\p \in \d p).$$
Under this new probability
$$ \w{\E} (\log m(\p))=c>0,$$
so under $\w{\P}, \ S_n= \sum_{i=1}^n \log m(\p_i)$ is a random walk with drift $c$ and $Z_n$ is a supercritical BPRE under  with survival probability $\w{p}>0$. Then, for every
$0<\epsilon<c$, \be \label{nv} \lim_{n\rightarrow \infty}
 \w{\P}_{z_0}\big( c-\epsilon\leq \log(Z_n)/n \leq c+\epsilon \big)=\w{p}>0.
\ee
Moreover, for every bounded measurable function $f$,
$$\E_{z_0}(f(Z_n))=[\E(m(\p)^{\lambda_c})]^n\w{\E}_{z_0}(\exp(- \lambda_c S_n) f (Z_n) ).$$
We will use the above with $f(z)=
\ind_{[c-\epsilon,c+\epsilon]}(\log(z)/n$ to obtain that, for every
$\eta >0$, \Bea
&&\P_{z_0}\big( c-\epsilon\leq \log(Z_n)/n \leq  c+\epsilon\big) \\
&=&  [\E ( m(\p)^{\lambda_c})]^n\exp \big( -n(\lambda_c c +\eta)
\big) \w{\E}_{z_0}\big(\exp(-\lambda_c S_n+ n(\lambda_c c+\eta))
f(Z_n)\big). \Eea
$\newline$
Now, under $\w{\P}$, $(-\lambda_c S_n+ n(\lambda_c
c+\eta))_{n\in\N}$ is a random walk with positive drift $\eta>0$
which tends to infinity as $n$ tends to infinity. By using
(\ref{nv}) we see that under $\w{\P}$ $f(Z_n) \to 1$ almost surely so that
$$\w{\P}_{z_0}\big(\liminf_{n\rightarrow \infty} \exp(-\lambda_c S_n+ n(\lambda_c c+\eta)) f(Z_n)
=  \infty \big) = \w{p}. $$ This ensures, by Fatou's lemma,
$$\liminf_{n\rightarrow  \infty}
\w{\E}_{z_0}\big(\exp(-\lambda_c S_n+ n(\lambda_c c+\eta))
f(Z_n) \big)=\infty.$$
And since $\E_{z_0} (f(Z_n)) = \E(m(\p)^{\lambda_c})^n \w{\E}_{z_0} (exp(-\lambda_c S_n) f(Z_n))$ we get
\begin{align*}
\liminf_{n\rightarrow \infty}  \frac1n \log(\P_{z_0}\big(c-\epsilon\leq
\log(Z_n) \leq  c+\epsilon \big) \big)/n &\geq \log [\E
(m(\p)^{\lambda_c})] - \lambda_c c -\eta \\ &= -\psi(c)-\eta.
\end{align*}

Letting $\eta\rightarrow 0$  gives
$$\liminf_{n\rightarrow \infty} \frac1n\log\big(\P_{z_0}\big(c-\epsilon\leq \log(Z_n) /n  \leq  c+\epsilon  \big)  \big) \geq
-\psi(c),$$
which completes the proof.
\end{proof}

\section{Key lemmas for lower deviation}
\label{S:key lemmas}

\subsection{The function $\chi$}
Observe that we have the following non-asymptotic bound \cite{dembo zeitouni}:
If $c\le \bar L$,
\begin{align}\label{E:simple large dev bound lower}
 \forall n \in \N: \P(S_n \le n c) \le \exp(-n \psi(c))
\end{align}
and if $c \ge \bar L$,
\begin{align}\label{E:simple large dev bound upper}
 \forall n \in \N: \P(S_n \ge n c) \le \exp(-n \psi(c)).
\end{align}
$\newline$

We recall that
\Bea
\chi  (c) &:=& \inf_{  t \in [0,1]} \{ - t \log (\E(\p (1))) + (1-t) \psi(c/(1-t))\}.
%& =&\inf_{\alpha \in [0,c/m]}
%\{ - \alpha \log (\E(\p (1)))+ (1-\alpha) \psi(c/(1-\alpha))\}.
\Eea
\begin{Lem}\label{L:def de chi}
There exists a unique $t_c \in  [0,1]$ such that
$$\chi(c) = -t_c \log (\E(\p(1)))+ (1-t_c) \psi(c/(1-t_c)),$$
and $t_c\in [0,1-c/\bar{L}]$.
\end{Lem}
\begin{proof}

Put $\rho := -\log (\E(\p (1)))$ and $v(t) :=   \rho t +(1-t) \psi(c/(1-t)).$
Then we have $v'(t)=\rho - \psi (c/(1-t)) + \frac{c}{1-t} \psi'(\frac{c}{1-t})$ and if we let $y =\frac{c}{1-t}$ we thus want to solve the equation
\begin{align*}
 0 &= v'(t)= \rho - \psi (y) +y \psi'(y)
\end{align*}

Assume that $v'(t)=0$ has two solutions $t_1 < t_2$ both in $[0,1],$ then there exists $t_3 \in (t_1, t_2)$ such that $v''(t_3) = 0$, i.e.
$$
0=-\psi'(y_3) + \psi'(y_3) +y_3 \psi''(y_3), \quad \t{where} \ y_3=\frac{c}{1-t_3}.
$$
That is impossible since $\psi'' >0.$ Adding that $v'(1-c/\bar{L})=\rho>0$ completes the proof.
\end{proof}
$\newline$

\subsection{The cost of staying bounded}

We start with the following elementary result, which says that staying bounded has the same exponential cost than staying fixed at 1.

\begin{Lem}\label{L:staying bounded}
For every $N\geq 1$,
$$\lim_{n\rightarrow \infty}\log(\P(Z_n\leq N))/n=\log(\E(\p(1))).$$
Moreover, if $\E(\p(1))>0$, then for every fixed $N$ there is a constant $C$ such that for every $n \in \N$,
$$\P(Z_n\leq N) \le C n^{N}\E(\p(1))^{n+1}.$$
\end{Lem}

\begin{proof}
We call $(N_i)_{ i\geq 1}$ the number of offspring of a random lineage. More explicitly, we call $N_0$ the size  of the offspring of the ancestor in generation $0$ and pick uniformly one individual among this offspring. We call $N_1$ the size of the offspring of this individual and so on...

Note that $(N_i)_{i\geq 1}$ are iid with common ditribution $\P(N=k)=\E(\p(k))$.
Hence, for every $n\geq N$, recalling that $\P(p(0)=0)=1$,
\begin{align*}
 \P(Z_n \leq N) &\le \P( \text{less than $N$ of the $(N_i)_{0\leq i\leq n-1}$ are} >1  ) \\
&\le   \sum_{k=0}^{N} \left(\begin{array}{c} n \\ k \end{array} \right) (1-\E(\p(1)))^k \E(\p (1))^{n-k} \\
& \le (N+1) n^{N} \E(\p (1))^{n-N}.
\end{align*}
Adding that
$$\P(Z_n\leq N)\geq \P(Z_n=1)=\E(\p(1))^n,$$
allows us to conclude.
\end{proof}

Our proof actually shows the stronger $$\lim_{n\rightarrow \infty}\log(\P(Z_n\leq n^a))/n=\log(\E(\p(1))),$$ for $a \in (0,1).$
$\newline$

\subsection{The cost of deviating from the environments}

The aim of this section is to show that  once the process ``takes off'' (i.e. once the population passes a certain threshold), it has to follow the products of the means of the environments sequence.

\begin{Lem}
\label{L:cost of dev from env} Assuming $\H$, for all  $\epsilon>0$
and $\eta>0$, there exist $N,D \in \N$ such that for all $z_0\geq
N$ and  $n\in\N$,
$$
\P_{z_0}
(Z_n\leq z_0 \exp(S_n-n\epsilon) \ \vert \ (\p_i)_{i=0}^{n-1})\leq
D\eta^n \quad \t{a.s.}
$$
so that
$$
\P_{z_0} (Z_n\leq \exp(S_n-n\epsilon))\leq  D\eta^n.
$$
\end{Lem}

 Define for every $0\leq i\leq n-1$,
$$R_i:=Z_{i+1}/Z_i,$$
so that
$$ Z_n= Z_0 \Pi_{i=0}^{n-1} R_i.$$

For all $\lambda\geq 0$, $q\in\N$ and $0\leq i\leq n-1$ define the
function
$$\Lambda_{q}(\lambda,p) := \E\big(\exp(\lambda[L_i -\epsilon -\log(R_i)]) \ \vert \
\p_{i}=p, \ Z_{i}=q \big),$$ (this quantity does not depend on $i$ by Markov property) and

$$M_{N} (\lambda, p):=\sup_{q\geq N}  \Lambda_{q}(\lambda,p). $$

The proof will use the following Lemma, the proof of which is given at the end of this section.
\begin{Lem}\label{L:majoration de M}
Fix $\epsilon >0,$ there exist $\alpha \in (0,1),\lambda_0 \in(0,1)$
and $N\in \N$ such that
$$M_{N } (\lambda_0, \p) \le 1-\alpha \qquad \text{a.s.}$$ where
$\p$ is a random probability with law $\mu.$
\end{Lem}
$\newline$

We proceed with the proof of Lemma \ref{L:cost of dev from env} assuming that the above result holds.
\begin{proof}

Let us fix $\epsilon >0$ and $k \in \N$ and let us show that $\exists \alpha \in (0,1), N\in\N, C>0$ such that
$\forall n \in \N, z_0 \geq N$
\begin{equation} \label{E:preuve de cost of dev 1}
 \P_{z_0} (Z_n \le k z_0 \exp (S_n - n \epsilon) |  \ (\p_i)_{i=0}^{n-1}) \leq  C (1-\alpha)^n.
\end{equation}

For every  $\lambda >0$,
\begin{align}
\label{decsim}
& \nonumber \P_{z_0} \big(Z_n\leq  k z_0 \exp(S_n-n\epsilon)\ \vert \ (\p_i)_{i=0}^{n-1} \big)  \qquad \qquad \\
& \nonumber \qquad =  \P_{z_0} \big( z_0 \Pi_{i=0}^{n-1} R_i \leq k z_0 \exp(\sum_{i=0}^{n-1} [L_i-\epsilon]) \ \vert \ (\p_i)_{i=0}^{n-1}\big) \\
& \nonumber \qquad =\P_{z_0}\big(\sum_{i=0}^{n-1}
\log (R_i) \leq \log k + \sum_{i=0}^{n-1} [L_i -\epsilon
]  \ \vert \ (\p_i)_{i=0}^{n-1}\big)\\
& \qquad \le  k^{\lambda} \
 \E_{z_0}\big(\exp \{ \lambda \sum_{i=0}^{n-1} [L_i -\epsilon -\log R_i
]\} \ \vert \ (\p_i)_{i=0}^{n-1}\big).\nonumber
\end{align}

Observe that conditionally on $\p_{j}, R_j$ depends on
$(\p_i)_{i=0}^{j}$ and $(Z_0, R_0, R_1, \ldots , R_{j-1})$ only
through $Z_j.$  Furthermore, under $\P_{z_0}$ we have that almost
surely $\forall n \in \N :Z_n \geq z_0$ since $\P(\p(0)>0)=0$. Hence
we get for every $\lambda\geq 0$,
\Bea
&&\P_{z_0} \big(Z_n\leq k z_0 \exp(S_n-n\epsilon)\ \vert \ (\p_i)_{i=0}^{n-1}\big)\\
&& \qquad \leq k^{\lambda} \ \E_{z_0}\big(\exp(\lambda\sum_{i=0}^{n-1} [L_i -\epsilon
-\log(R_i)]) \ \big\vert \ (\p_i)_{i=0}^{n-1}\big)\\
&& \qquad\leq  k^{\lambda} \
\E_{z_0}\Big\{\exp(\lambda\sum_{i=0}^{n-2} [L_i -\epsilon
-\log(R_i)])  \\ && \qquad \qquad \times \E_{z_0} \big[ \exp(\lambda
[L_{n-1} -\epsilon
-\log(R_{n-1})]) \ | \p_{n-1}, Z_{n-1} \big] \ \Big\vert \ (\p_i)_{i=0}^{n-1}\Big\}\\
& &  \qquad \leq  k^{\lambda} \  \E_{z_0}\big(\exp(\lambda\sum_{i=0}^{n-2} [L_i -\epsilon
-\log(R_i)]) \ \big\vert \ (\p_i)_{i=0}^{n-2}\big) M_{z_0}(\lambda,\p_{n-1}) \\
& &  \qquad \leq  ... \\
& &  \qquad \leq  k^{\lambda} \  \Pi_{i=0}^{n-1} M_{z_0}(\lambda,
\p_i). \Eea

From Lemma \ref{L:majoration de M} we can find  $\alpha \in (0,1),
\lambda_0 \in (0,1)$  and $\exists N \in \N$ such that almost surely
$\forall i \in \N, \, M_{N}(\lambda_0,\p_i) \leq 1-\alpha.$ Hence,
for all $z_0 \ge N$ we have, \be \label{lamajorat} \P_{z_0} (Z_n
\leq k z_0 \exp(S_n-n\epsilon) \ \vert \ (\p_i)_{i=1}^n)\leq
k^{\lambda_0} \prod_{i=1}^n M_{z_0}(\lambda_0, \p_i) \leq
k^{\lambda_0} (1-\alpha)^n \quad \t{a.s.} \ee which proves
(\ref{E:preuve de cost of dev 1}) and we can now prove Lemma \ref{L:cost
of dev from env}. Let $\eta>0$ and fix  $k\in \N$ such that
$(1-\alpha)^k\leq \eta$. Then for every $z_0\geq kN$, using
successively that, conditionally on $(\p_i)_{i=0}^{n-1}$,  $Z_n$
increases when the initial number of individual increases,  and that
$Z_n$ starting from  a population of $k$ groups of $[z_0/k]$
individuals is the sum of $k$ iid variables distributed as
$\P_{[z_0/k]}(Z_n \in .)$, we get \Bea
&& \P_{z_0}(Z_n\leq z_0 \exp(S_n-n\epsilon) \ \vert \ (\p_i)_{i=0}^{n-1}) \\
&& \qquad \leq  \P_{k[z_0/k]}(Z_n\leq z_0 \exp(S_n-n\epsilon) \ \vert \ (\p_i)_{i=0}^{n-1})\\
&& \qquad \leq  \P_{[z_0/k]}(Z_n\leq  z_0  \exp(S_n-n\epsilon) \ \vert \ (\p_i)_{i=0}^{n-1})^k\\
&& \qquad \leq  \P_{[z_0/k]}(Z_n\leq  (k+1)[z_0/k] \exp(S_n-n\epsilon) \ \vert \ (\p_i)_{i=0}^{n-1})^k\\
&& \qquad \leq  (k+1)^{\lambda_0k}(1-\alpha)^{kn}, \Eea using
(\ref{lamajorat}). This  completes the proof of  Lemma \ref{L:cost
of dev from env} with $D=(k+1)^{\lambda_0 k}$.
\end{proof}
$\newline$

We now prove Lemma \ref{L:majoration de M}.
\begin{proof}
Observe that the $(\Lambda_{q}(\lambda ,\p_i)_{i\in\N}$ are iid with common
distribution

$$\Lambda_q(\lambda) =
\Lambda_{\lambda}(\p_0,q) = \E(\exp(\lambda [L_0 -\epsilon -\log
R_0]) \ | \p_0, \ Z_0=q).
$$

By Taylor's formula, for every $\lambda\geq 0$, there exists
$c_\lambda\in [0,\lambda]$ such that \be \label{eqone} \Lambda_{q}
(\lambda)=1+\lambda\E\big( L_0 -\epsilon -\log(R_0) \ \vert \ \p_0,
\ Z_0=q \big)+\lambda^2 \Lambda_{q} ''(c_{\lambda}). \ee

Let us first show that we can find $N$ such that for everey $q \ge
N$, \be \label{eqtwo} \E \big(L_0 -\epsilon -\log(R_0) \ \vert \
\p_0, \ Z_{0}=q \big) \le -\epsilon/2. \ee Observe that  $m :=
m(\p_0)=\exp(L_0)$ and $R_0$ are both bigger than 1 almost surely so $|\log
(R_0) - L_0|< |R_0 - m|$ and hence
\begin{align*}
 \big| \E\big[ \log(R_0) - L_0 \ | \ \p_0 , Z_0=q  \big]  \big| & \le \big| \E\big[ R_0 - m \ | \ \p_0 , Z_0=q  \big]  \big| \\
 &\le \E\big[ |R_0 - m| \ \big| \ \p_0 , Z_0=q  \big]   \\
 & \le \E\big[ (R_0 - m)^2 \ \big| \ \p_0 , Z_0=q  \big]^{1/2}   \\
 & = \text{Var} (R_0 \ \big| \ \p_0, Z_0=q  \big)^{1/2}   \\
 &= \left(\frac1q \text{Var}_{\p_i} \right)^{1/2},
\end{align*}
using that conditionally on $Z_0=q$ and $\p_0=p$,
$R_0=q^{-1}\sum_{j=1}^q X_j$ where $(X_j)_{j=1, \ldots ,q}$ are iid
with common law $p$. By hypothesis $\H, \text{Var}_{\p_0}$ is
bounded so there exists $N\in\N$ such that for every $q\geq N$,
$$\big| \E\big[ \log(R_0) - L_0 | \p_0 , Z_0=q  \big]  \big| \le \epsilon/2 \qquad \text{ a.s}.$$

To bound $\Lambda''_{q}(\lambda)$, observe that for any $\lambda \in
[0,1]$, \Bea \Lambda_{q} ''(\lambda) &=& \E\Big[  (L_0 - \epsilon -
\log R_0)^2e^{\lambda(L_0 - \epsilon-  \log R_0)} \big| \ \p_0, \ Z_0=q
\Big]
\\
&\leq &\E\Big[ (\log  m -\epsilon -\log(R_0))^2 m \ \big\vert \
\p_0, \ Z_{0}=q \Big]
\\
& \leq  & m \E\Big[ 4((\log m)^2 +\epsilon^2+\log(R_0)^2 ) \ \big
\vert \ \p_0, \ Z_{0}=q \Big].
\\
&\leq & 4A \Big[ \t{esssup}((\log m(\p))^2) +\epsilon^2 + \E\big(
\log(R_0)^2 \ \vert \ \p_0, \ Z_{0}=q \big)\Big] \qquad \text{a.s.},
\Eea where $A$ is the constant from $\H$. Then, denoting by
$(N_j)_{j\in\N}$  iid r.v. with common law $\p_0$,
 observe that
\begin{align*}
  \E\big( \log(R_0)^2 \ \vert \ \p_0, \ Z_{0}=q \big) &\le  \E\big( (R_0-1)^2
\ \vert \ \p_0, \ Z_{i}=q \big)
\\ &\le 1 + \E\big( R_0^2  \ \vert \ \p_0, \ Z_{0}=q \big) -2 \E\big(
R_0 \ \vert \ \p_0, \ Z_{0}=q \big) \\ &\le 1 +  \frac2{q^2} \E\big(
\sum_{j=1}^q N_j^2  \ \vert \ \p_0, \ Z_{0}=q \big) -2 m \qquad
\\ &\le 1 +  \frac2{q} B \qquad \text{ a.s.},
\end{align*}
where $B$ is the constant from $\H$. So we can conclude that for all
$\lambda \in [0,1]$ and $q \in \N$
\begin{align*}
\Lambda_{q} ''(\lambda)  &\le M \qquad \text{ a.s.},
\end{align*}
where $M$ is a finite constant. Then,  for all $q\geq N$ and
$\lambda \in [0,1]$,
$$\Lambda_{q} (\lambda)\leq 1-\lambda \epsilon/2 +\lambda^2 M \qquad \text{ a.s.},$$ and thus
$$
M_{N}(\lambda,\p_0) \leq 1-\lambda \epsilon/2 +\lambda^2 M \qquad
\text{ a.s. }.
$$

Choose now $\lambda _0 \in (0,1]$ small enough such that $ \lambda_0\epsilon/2  -\lambda_0^2 M =\alpha>0,$ then
 $M_{N} (\lambda_0, \p_0)\leq 1-\alpha$ $a.s.$ This ends up the proof of Lemma \ref{L:majoration de M}.
\end{proof}

\section{Proof of Theorem \ref{T:trajectoire gauche}}
\label{S:proof}

For each $c<\bar{L}$, we start by giving the rate function for lower deviations and we prove that $(Z_{[nt]})_{t\in[0,1]}$ begins to take large values at time $t_c$. We then show thatno jump occur at time $t_c$
and that $(\log(Z_{[nt]})/n)_{t\in[t_c,1]}$ grows linearly to complete the proof of Theorem \ref{T:trajectoire gauche}. \\

\subsection{Deviation cost and take-off point}
We consider the first time at which the
population reaches the threshold $N$
$$\tau(N) := \inf \{ k : Z_k  >N\}, \quad \tau_n(N)=\min(\tau(N),n).$$
Recalling that by Lemma \ref{L:def de chi},
$$\chi  (c) = \inf_{  t \in [0,1-c/\bar{L}]} \{ - t \log (\E(\p (1))) + (1-t) \psi(c/(1-t))\}$$
and that $t_c$ is the unique minimizer,we have the following statement. \\
%& =&\inf_{\alpha \in [0,c/m]}
%\{ - \alpha \log (\E(\p (1)))+ (1-\alpha) \psi(c/(1-\alpha))\}.

\begin{Prop}\label{P:total cost et take-off}
For each $c <\bar L$,
 $$ \lim_{n\rightarrow \infty} -\frac1n \log \P(\log(Z_n)/n\leq c)= \chi(c).$$
Furthermore,  for $N$ large
enough, conditionally on $Z_n \le e^{cn}$,
$$\tau_n(N)/n \stackrel{n\rightarrow \infty}{\longrightarrow } t_c \quad \t{in} \  \P.$$
\end{Prop}

%xyz faire un lemme plus qu'une prop ?

For the proof, we need the following lemma, which tells us that once the
population is above $N$, the cost of a deviation for $(Z_n)_{n\geq 0}$ is
simply the cost the necessary sequence of environments, i.e. the deviation cost for the random walk  $(S_n)_{n\geq 0}.$\\
By decomposing the total probability cost of reaching $nc$ in two pieces (staying bounded until time $nt$ and then having  $(S_n-S_{t_cn})\simeq nc$) and then minimizing over $t$ gives us the correct rate function. The unicity of this minimizer $t_c$ ensures then that the take-off point $\tau_n(N)/n$ converges
to $t_c$. \\

\begin{Lem}\label{P:cost with large pop}
Assume $\H$.
\begin{itemize}
 \item[(i)] For each $\eta>0, \epsilon>0$, there exists $D,N\in\N$ such that for all
 $c\geq 0, \ z_0 \ge N$ and $ n\in\N$,
 \Bea
 \P_{z_0} \big( Z_n \leq  z_0 \exp(cn) \big) \le D( \eta^n +\exp (-n\psi^*(c+ \epsilon ))),
 \Eea
where $\psi^*(x)=\psi(x)$ for $x\leq \bar{L}$ and $\psi^*(x)=0$ for $x\geq \bar{L}$. \\
\item[(ii)] For every $ \epsilon >0$ and for every $c_0 \le \bar L -\epsilon$ such that
$\psi(c_0)<\infty$, there exists $N$ such that for all $z_0\geq N$
and $c \in [c_0, \bar L -\epsilon]$,
$$
\liminf_{n\rightarrow \infty} - \frac1n \log \P_{z_0} ( Z_n \le z_0 e^{cn})  \ge \psi(c
+\epsilon)
$$
and
$$
\limsup_{n\rightarrow \infty} - \frac1n \log \P_{z_0} ( Z_n \le z_0 e^{cn})  \le
\psi(c).
$$
\end{itemize}
\end{Lem}

\begin{proof}
For each $z_0 \in \N, c \le \bar L, n \in \N$ and $\epsilon>0$,
\Bea
&&\P_{z_0}\big( Z_n\leq z_0 \exp(cn) \big) \\
&\qquad & \ \ \leq\P_{z_0}\big(Z_n\leq z_0\exp(cn), \ S_n-n\epsilon\geq cn \big)+ \P_{z_0}\big( S_n-n\epsilon\leq cn \big) \\
&\qquad & \ \ \leq \P_{z_0}\big( Z_n\leq z_0\exp(S_n-n\epsilon)\big)
+\P_{z_0}\big( S_n \leq (c+\epsilon) n\big). \Eea Let $\eta>0$, then
by Lemma \ref{L:cost of dev from env} and (\ref{E:simple large dev
bound lower}), there $\exists D,N:=D(\epsilon, \eta),
N(\epsilon,\eta)$ such that for all $c\leq \bar{L}-\epsilon$,
$z_0\geq N$,
$$
\P_{z_0}\big( Z_n\leq z_0 \exp(cn ) \big) \\
\leq D\eta^n+\exp(-n\psi^*(c+\epsilon)).
$$
which yields (i). \\

The first part of (ii) is an easy consequence of (i) by taking $\eta
< \inf \{ \exp (-\psi(c)), c \in [c_0, \bar L - \epsilon]\}$. The
second part comes directly from Proposition \ref{P:minoration}.
\end{proof}
$\newline$

\begin{proof}[Proof of Proposition \ref{P:total cost et take-off}]

If $\E(\p (1))=0$ then $\chi(c) = \psi(c)$ and $t_c=0$. Noting that $Z_n\geq 2^n \ a.s.$ gives directly
the second part of the lemma, while  the first part follows essentially from
Lemma \ref{P:cost with large pop} (ii).  \\

We suppose now that $\E(\p (1))>0.$ For each $c \le \bar L$
and   $i=1,\ldots, n$, we have for every $z_0\in\N$,
 \Bea
\P(\tau_n(N)=i, \  \  Z_n\leq \exp(cn)) & \leq &\P(Z_{i-1}\leq N)\P_N(Z_{n-i}\leq \exp(cn)) \\
& \leq  & \P(Z_{i-1}\leq N) \P_N(Z_{n-i}\leq N\exp(cn)).
\Eea
Using Lemma \ref{P:cost with large pop} and  Lemma \ref{L:staying bounded}, for all $\eta>0$ and $\epsilon>0$, there exists $N,M\in\N$ such that for
all $z_0 \ge N$,
 \Bea
\P(\tau_n(N)=i, \  \  Z_n\leq \exp(cn))
& \leq & M n^N \E(\p (1))^i
[\eta^{n-i}+\exp(-(n-i)\psi^*(cn/(n-i)+\epsilon)]. \Eea
Summing over $i$ leads to
\begin{align*}
\P(\log(Z_n)/n\leq c) &= \sum_{i=1}^n \P(\tau_n(N) =i, \ \log(Z_n)/n\leq c )\\
&\le \sum_{i=1}^n M n^N  \E(\p (1))^i
[\eta^{n-i}+\exp(-(n-i)\psi^*(cn/(n-i)+\epsilon)].
\end{align*}
Thus
\begin{align*}
 &\liminf_{n\rightarrow \infty} - \frac1n \log \P(\log(Z_n)/n\leq c) \\
& \qquad \ge \liminf_{n\rightarrow \infty} -\frac1n   \log \left( \E(\p (1))^i  [\eta^{n-i}+\exp(-(n-i)\psi^*(cn/(n-i)+\epsilon))]  \right) \\
& \qquad  \ge \liminf - \max_{i=1,\ldots, n}\left[ \frac{i}{n} \log \E(\p (1)) + \frac{n-i}{n} \log(\eta +  \exp(-\psi^*(cn/(n-i)+\epsilon)) )   \right] \\
& \qquad  \ge \inf_{t \in [0,1]}  \left\{ -t \log (\E(\p(1))) -(1-t)
\log (\eta +\exp(-\psi^*(c/(1-t) +\epsilon)))  \right\}.
\end{align*}
where, from the second to the third line, we have used that $a^n +
b^n \le (a+b)^n$ when $a, b \ge 0.$ Letting $\eta,\epsilon \to 0$,
we see that
\begin{align}
\liminf -  \frac1n \log \P(\log(Z_n)/n\leq c) & \ge \inf_{t \in
[0,1]}
\left\{ -t \log \E(\p(1)) + (1-t) \psi^*(c/(1-t))  \right\}\nonumber \\
&\geq  \inf_{t \in [0,1-c/\bar{L}]}
\left\{ -t \log \E(\p(1)) + (1-t) \psi(c/(1-t))  \right\} \nonumber \\
&=
%\inf_{t \in [0,1]}
%\left\{ -t \log \E(\p(1)) + (1-t) \psi(c/(1-t))  \right\} \\
%&=
\label{liminf} \chi(c),
\end{align}
where the two infimums coincide since $\psi^*(c/(1-t))=0$ as soon as
$t\geq 1- c/\bar{L}$. \\

More generally, given $0\le a < b \le 1$ it is an easy adaptation of
the above argument to show that
\bea
&&\liminf -\frac1n \log \P(\log(Z_n)/n\leq c, \tau_n(N)/n \in [a,b]) \nonumber \\
&& \qquad \qquad \qquad \label{ab}
\ge \inf_{t \in [a,b]\cap [0,1-c/\bar{L}]} \left\{ -t \log \E(\p(1)) + (1-t)
\psi(c/(1-t))  \right\} .
\eea

The upper bound is much easier since it is enough to exhibit a
trajectory having $\chi(c)$ as it asymptotic cost. By construction
it should be clear that
\begin{align*}
 \P(Z_{[t_c n]}=1 , Z_n \le e^{cn} ) &= \P(Z_{[t_c
 n]}=1)\P(Z_{n-[t_cn]} \le e^{cn})
\end{align*}
By Lemma \ref{L:staying bounded} and Proposition \ref{P:minoration},
\Bea
\limsup_n - \frac1n \log \P(Z_{[t_c n]}=1 , Z_n \le e^{cn} )
&\le & -t_c \log \E(\p(1)) + (1-t_c)
\psi(c/(1-t_c)) \\
&=&\chi(c).
\Eea
Combining this inequality with the lower bound given by (\ref{liminf}), this concludes the proof of the first point of Proposition \ref{P:total cost et take-off}. \\

For the convergence of $\tau_n(N)/n \to t_c,$ observe that by
Lemma \ref{L:def de chi}, $t_c$ is the unique minimizer of $t \in [0,1]
\mapsto \left\{ -t \log \E(\p(1)) + (1-t) \psi(c/(1-t))  \right\}$.
 Hence, if $t_c \not \in (a,b)$ we have
$$
 \inf_{t \in [a,b]\cap [0,1-c/\bar{L}]}
 \left\{ -t \log \E(\p(1)) + (1-t) \psi(c/(1-t))  \right\}\\
 >
\chi(c).
$$
This  means by (\ref{liminf}) and (\ref{ab}) that conditionally on $Z_n \le e^{cn}$ the
event $\tau_n(N)/n \in (a,b)$ becomes negligible with respect to the
event $\tau_n(N)/n \in [t_c -\epsilon, t_c +\epsilon]$ for
any $\epsilon>0$. This proves that $\tau_n(N)/n \to_p t_c.$
\end{proof}
$\newline$

\noindent Proposition \ref{P:total cost et take-off} already proves half of
Theorem \ref{T:trajectoire gauche}. We now proceed to the proof of the path behavior. Define a process $ t \mapsto Y^{(n)}(t)$ for $t \in [0,1]$ by
$$Y^{(n)}_t =
\frac1n \log (Z_{[nt]}).$$
 The second part of Theorem
\ref{T:trajectoire gauche} tells us that $Y^{(n)}(t)$ converges to
$f_c$ in probability in the sense of the uniform norm. To prove this
we need two more ingredients, first we need to show that after time
$\tau_n(N)/n\simeq t_c$ the trajectory of $Y^{(n)}(t)$ converges to a straight
line (this is the object of the following section \ref{S:straight
line}) and then that $Y^{(n)}$ does not jump  at time
$\tau_n(N)/n$ (in section \ref{S:fin preuve}).\\

%{\bf zzz attention : N depend de c, donc de $\tau_n(N)$ il y a un
%argument d'uniformité à fournir}

\subsection{Trajectories in large populations}\label{S:straight
line}

The following proposition shows that for a large enough initial
population and conditionally on $Y^{(n)}(1) < c$ the process
$Y^{(n)}$ converges to the deterministic function $t \mapsto ct.$

\begin{Prop}\label{P:la ligne droite}
For all $c<\bar{L}$  and  $\epsilon>0$, there exists $N \in \N$, such that for $z_0
\ge N$,
$$\lim_{n\rightarrow \infty} \P_{z_0}\big(\sup_{x \in [0,1]} \{  |Y^{(n)}(x) -cx|
 \geq \epsilon \ \vert \ Z_n\leq z_0 \exp(cn) \big)=0.$$
\end{Prop}
$\newline$

Before the proof, let us give a little heuristic of this result. Informally, for all $t\in (0,1)$ and $\epsilon>0$,
$$ \P_{z_0}(Y_t^{(n)}= c+\epsilon, \ Z_n \leq \exp(cn))= \P_{z_0}(Z_{[nt]}= \exp(tn(c+\epsilon)))\P_{\exp(tn(c+\epsilon))}
(Z_{n-[nt]}\leq \exp(cn)).$$
Then, for $z_0$ large enough,  Lemma \ref{P:cost with large pop} ensures that
\bea
&&\lim_{n\rightarrow \infty}-\log(\P_{z_0}(Y_t^{(n)}= c+\epsilon, \ Z_n\leq \exp(cn)))/n \nonumber \\
&& \qquad \qquad = t\psi(c+\epsilon)+(1-t)\psi(c-\epsilon t/(1-t)) \label{heuristic} \\
&& \qquad \qquad > \psi(c) \nonumber,
\eea
by strict convexity of $\psi$.  Adding  that $\limsup_n - \frac1n \log \P_{z_0} ( Z_n \le z_0 e^{cn})  \le \psi(c)$ by Proposition $1$  entails that the probability of this event becomes negligible as $n\rightarrow \infty$.

\begin{proof}
Observe that $\{\exists x \in [x_0,x_1] : Y^{(n)}(x) > cx +\epsilon
\} = \{\exists x \in [x_0,x_1] :  Y^{(n)}(x) \in ( cx +\epsilon,
\bar L x ] \}$, because a.s. $ t \mapsto Y^{(n)}(t)$ is an increasing function so that the only way $Y^{(n)}$ can
cross $x \mapsto \bar L x$ downward is continuously. Hence we can
divide the proof in the following steps :
\begin{itemize}
\item[(i)] There exists $0<x_0<x_1<1$ such that for every $\epsilon>0$ and
 for $z_0$ large
enough $\lim_{n\rightarrow \infty} \P_{z_0}\big(\sup_{x \not \in
[x_0,x_1]} \{ |Y^{(n)}(x) -cx|
 \geq \epsilon \ \vert \ Z_n\leq z_0 \exp(cn) \big)=0.$

\item[(ii)]  We show that for $z_0$ large enough  $\lim_{n\rightarrow \infty}
\P_{z_0}\big(\exists x \in [x_0,x_1] : cx +\epsilon \leq  Y^{(n)}(x) \leq \bar{L}x   \
\vert \ Z_n\leq z_0 \exp(cn) \big)=0$.

\item[(iii)] The fact that for $z_0$ large enough $\lim_{n\rightarrow \infty}
\P_{z_0}\big(\exists x \in [x_0,x_1] : Y^{(n)}(x) \leq  cx -\epsilon  \
\vert \ Z_n\leq z_0 \exp(cn) \big)=0$ then follows from the same
arguments as in (ii).
\end{itemize}

\medskip
We start by proving (ii) which is the key point. We can assume $\epsilon < (\bar L -c ) x_0$
and  $\epsilon < (\bar L -c )(1-x_1)$ and we define
$$R_c := \{ (x,y) : x \in [x_0,x_1],  y \in [cx+\epsilon, cx]\}.$$
We know
from Lemma \ref{P:cost with large pop} that $ \limsup_{n\rightarrow \infty}
-\frac1n \log \P_{z_0} ( Z_n\leq z_0 \exp(cn) ) \le \psi(c)$ (for
$z_0$ large enough). Hence, we will have proved the result if we
show that for $z_0$ large enough
\begin{align}\label{E:ineq pour (ii)}
\liminf_{n\rightarrow \infty} -\frac1n \log \P_{z_0}\big(\exists x \in [x_0,x_1] :
(x,Y^{(n)}(x)) \in R_c ,  Z_n\leq z_0e^{cn} \big) > \psi(c).
\end{align}

 Lemma \ref{P:cost with large pop} or heuristic (\ref{heuristic}) suggest that
 the asymptotic cost of the event
$\{ Y^{(n)}(x) =y, Y^{(n)}(1) <c \}$ is given  by the map
$$
 x,y \in
[0,1] \mapsto  x \psi(y/x) + (1-x) \psi((c-y)/(1-x)).
$$
More precisely, consider a cell $\theta=[x_l,x_r] \times [y_d, y_l]
\subset R_c$ and define for every $\eta\geq 0$,
$$C_{c,\eta}(\theta)  := x_l\psi(y_d/x_l+\eta) +(1-x_r)
\psi((c-y_d)/(1-x_r)+\eta).$$
 Observe that
$$
\{ \exists x : (x,Y^{(n)}(x) ) \in \theta \} \subset  \{Y^{(n)}(x_l) \le
y_l\} \cap \{ Y^{(n)}(x_d) \geq y_d \},
$$
so using the Markov property and the fact that $z_0 \mapsto
\P_{z_0}(Y^{(n)}(1) \le c)$ is decreasing
\begin{align*}
& \P_{z_0} ( \exists x : (x,Y^{(n)}(x) ) \in \theta, \ Y^{(n)}(1)  \le
c) \\ & \qquad \leq  \P_{z_0}(Y^{(n)}(x_l) \le y_l ,  Y^{(n)}(x_r) \geq y_d,
Y^{(n)} (1) - Y^{(n)}(x_r) \le c - Y^{(n)}(x_r))
\\ & \qquad \le \P_{z_0}(Y^{(n)}(x_l) \le y_l ) \sup_{y \ge y_d} \P_{[\exp{ny}]}( Y^{(n)}(1-x_r) \le (c - y)/(1-x_r))
\\ & \qquad \le \P_{z_0}(Y^{(n)}(x_l) \le y_l ) \P_{[\exp{ny_d}]}( Y^{(n)}(1-x_r) \le (c - y_d)/(1-x_r))
\end{align*}
Hence, using Lemma \ref{P:cost with large pop} (ii), we see that for every $\eta>0$ small enough, there exists  $
N(\eta,\theta)$ large enough such that for every $z_0\geq N(\eta,\theta)$,
\begin{align}
\nonumber
\liminf_{n\rightarrow \infty} -\frac1n \log \P_{z_0}(\exists x\in [0,1] : (x,Y^{(n)}(x)) \in \theta, Y^{(n)}(1)<c )\geq C_{c,\eta}(\theta).
\end{align}
By continuity of $\eta,\theta \rightarrow C_{\eta, c}(\theta)$,
$$\inf_{\substack{\theta \subset R_c, \\  \t{diam} (\theta)\leq \delta}} \big\{ C_{\eta,c}(\theta) \big\} \stackrel{\delta, \eta \rightarrow 0}{\longrightarrow } \inf_{z \in R_c} \{C_{0,c}(\{z\})\}.$$
Moreover for every $z=(x,y) \in R_c$, $x\in[x_0,x_1]$ and $y/x>c$, so by strict convexity of $\psi$,
$$C_{0,c}(\{z\})=x\psi(y/x) +(1-x)
\psi((c-y)/(1-x))>\psi(c).$$
Then $\inf_{z \in R_c} \{C_{0,c}(\{z\})\}>\psi(c)$, and there exists $\delta_0>0$ and $\eta>0$ such that
for every cell $\theta$ whose diameter is less than $\delta_0$, for every $z_0\geq N(\eta,\theta)$,
\be
\label{E:cost of theta}
\liminf_{n\rightarrow \infty} -\frac1n \log \P_{z_0}(\exists x\in [0,1] : (x,Y^{(n)}(x)) \in \theta, Y^{(n)}(1)<c )>\psi(c).
\ee

Fix an arbitrary region $R \subset R^\circ_c$ included in the
interior of $R_c$. We can chose $0<\delta\leq \delta_0$
such that there is a cover of $R$ by the union of a finite
collection $\mathcal{K}$ of rectangular regions $[x(i),x(i+1)
]\times [y(j),y(j+1)]$ with $i \in \{1,\ldots , N_{\delta}\}$ and $j
\in \{1,\ldots, N(i)\}$  such that their diameter is never more than
$\delta.$  \\
Observe that for every $z_0\geq 1$,
\begin{align*}
\P_{z_0}( \exists x : (x,Y^{(n)}(x)) \in R' , Y^{(n)} \le c)
&\le \sum_{\theta \in \mathcal{K}} \P_{z_0}( \exists x : (x,Y^{(n)}(x)) \in \theta , Y^{(n)} \le c)  \\
&\le |\mathcal{K}| \sup_{\theta \in \mathcal{K}} \P_{z_0}( \exists x
: (x,Y^{(n)}(x)) \in \theta , Y^{(n)} \le c).
\end{align*}
Then using (\ref{E:cost of theta}) simultaneously for each cell $\theta \in \mathcal{K}$,
we conclude that for every $z_0\geq N=\max \{N(\theta,\eta) : \theta \in \mathcal{K}\}$,
\begin{align*}
& \liminf_{n\rightarrow \infty} -\frac1n \log  \P_{z_0}( \exists x : (x,Y^{(n)}(x)) \in
R' , Y^{(n)} \le c) \\ &\qquad = \min_{\theta \in \mathcal{K}}
\liminf_{n\rightarrow \infty}  - \frac1n\log  \P_{z_0}( \exists x : (x,Y^{(n)}(x)) \in
\theta , Y^{(n)} \le c) \\ &\qquad > \psi(c).
\end{align*}
As $R'$ is arbitrary in the interior of $R_c$ this concludes the
proof of (\ref{E:ineq pour (ii)}) and  (ii). \\

\medskip

Let us now proceed with the proof of (i). Recall that under hypothesis $\H, \P(L
>\log A) =0 $ (i.e. the support of $L$ is bounded by $ \log A.$) Fix
$\zeta >0$ and take $x_0,x_1$ such that  $\epsilon / x_0 >A+ \zeta,
x_0 c < \epsilon$  and $ c + \epsilon/(1- x_1) >A+\zeta, \epsilon > c
(1-x_1).$

\begin{align*}
&\P_{z_0} \big( \exists x \le x_0 : |Y^{(n)}(x) - cx| >\epsilon, Y^{(n)}(1) \leq  c  \big) \\
& \le \P_{z_0} \big( \exists x \le x_0 : Y^{(n)}(x) - cx >\epsilon\big) \\
& \le \P_{z_0} \big(   Y^{(n)}(x_0) >\epsilon\big) \\
& \le \P_{z_0} \big(   Y^{(n)}(x_0) >  x_0(A+\zeta) \big) \\
& \le \P_{z_0} \big(   \log( Z_{[nx_0]}) >  S_{[n x_0]} + \zeta n x_0
\big)
\end{align*}
since $n x_0 (A+\zeta) - S_{n x_0} > \zeta n x_0.$ Hence this
requires a ``deviation from the environments" and by Lemma
\ref{L:cost of dev from env} for $\eta$ fixed, there exists $D\geq 0$ such that for $z_0$
large enough,
$$
P_{z_0} \big( \exists x \le x_0 : |Y^{(n)}(x) - cx| >\epsilon,
Y^{(n)}(1) < c + \log z_0/n \big) \leq D \eta^{nx_0}.
$$
Picking $\eta$ small enough ensures that this is in $o(\exp (-n
\psi(c))).$ The argument for the $[x_1, 1]$ part of the interval is
similar. Thus, recalling that $ \limsup_{n\rightarrow \infty}
-\frac1n \log \P_{z_0} ( Z_n\leq z_0 \exp(cn) ) \le \psi(c)$ for $z_0$ large enough, we get $(i)$.
\end{proof}
$\newline$

We can also prove the following stronger result. For every $c<\bar{L}$, for every $\epsilon>0$, there exists $N \in \N$ and $\alpha>0$, such that for $z_0
\ge N$,
\be
\label{limunif}
\lim_{n\rightarrow \infty} \sup_{c'\in [c-\alpha,c+\alpha]} \P_{z_0}\big(\sup_{x \in [0,1]} \{  |Y^{(n)}(x) -c'x|
 \geq \epsilon \ \vert \ Z_n\leq z_0 \exp(c'n) \big)=0.
 \ee
Indeed the proof of Lemma \ref{P:cost with large pop} (ii) also ensures that
for every $ \epsilon >0$ and for every $c_0 \le \bar L -\epsilon$ such that
$\psi(c_0)<\infty$
there exists $N$ such that for $z_0\geq N$,
$$
\liminf_{n\rightarrow \infty} \inf_{c\in [c_0,\bar{L}]} \{ - \frac1n \log \P_{z_0} ( Z_n \le z_0 e^{cn})
- \psi(c
+\epsilon)\}\geq 0.
$$
Then, following the proof of (ii) above with now
$$\inf_{c\in [c_0,\bar{L}]} \{\inf_{\substack{\theta \subset R_c, \\  \t{diam} (\theta)\leq \delta}} \big\{ C_{\eta,c}(\theta) \big\} - \inf_{z \in R_c} \{C_{0,c}(\{z\})\} \} \stackrel{\delta, \eta \rightarrow 0}{\longrightarrow}0,$$
there exists $\delta_0>0$ and $\eta>0$ such that
for every cell $\theta$ whose diameter is less than $\delta_0$, for every $z_0\geq N(\eta,\theta)$, (\ref{E:cost of theta})
becomes
$$
\beta=\liminf_{n\rightarrow \infty} \inf_{c\in[c_0,\bar{L}]}\{-\frac1n \log \P_{z_0}(\exists x\in [0,1] : (x,Y^{(n)}(x)) \in \theta, Y^{(n)}(1)<c )-\psi(c)\}>0.
$$
Moreover for every $\epsilon>0$,
\Bea
\limsup_{n\rightarrow \infty} \sup_{c'\in[c-\alpha,c+\alpha]}
-\frac1n \log \P_{z_0}(Z_n\leq \exp(c'n)) &\leq & \limsup_{n\rightarrow \infty}
-\frac1n \log \P_{z_0}(Z_n\leq \exp((c-\alpha)n))\\
&=&\psi(c-\alpha).
\Eea
Putting the two last  inequalities together with  $\alpha>0$ such that $\psi(c-\alpha)\leq \psi(c+\alpha)+\beta$
and $[c-\alpha,c+\alpha]\subset [c_0,\bar{L}-\epsilon]$ gives (\ref{limunif}).\\

\subsection{End of the proof of Theorem \ref{T:trajectoire gauche}}
\label{S:fin preuve}
We  begin to prove that $(Z_{n})_{n\in\N}$ does not make a big jump when it goes up
to $N$ in the following sense.
\begin{Lem}
\label{saut}
For every $c<\bar{L}$ and $N\in\N$,
$$\sup_{n\in\N} \P(Z_{\tau_n(N)}\geq N+M \ \vert Z_n\leq e^{cn} )  \stackrel{M\rightarrow \infty}{\longrightarrow } 0.$$
\end{Lem}

\begin{proof}
By the Markov property, for any $b$ and $a \le N$ fixed,
\begin{align*}
 & \P(Z_{\tau_n(N)}\geq N+M \ \big \vert \ Z_n\leq e^{cn}, \tau_n(N) =b, Z_ {\tau_n(N)-1} =a
 ) \\ & \qquad = \P_a(Z_{1}\geq N+M \ \big\vert \ Z_{n-b}\leq e^{cn}) \\ & \qquad \le\P_N(Z_{1}\geq N+M \ \big\vert \ Z_{n-b}\leq
 e^{cn})
 \\ & \qquad = \frac{\P_N(Z_{n-b}\leq
 e^{cn} \ \big\vert \ Z_{1}\geq N+M )) \P_N(Z_1 \geq N+M)  }{\P_N(Z_{n-b} \le e^{cn})}
\end{align*}
by Bayes' formula. Observe that
\begin{align*}
\P_N(Z_{n-b}\leq
 e^{cn} \ \big\vert \ Z_{1}\geq N+M )) &\le\P_N(Z_{n-b}\leq
 e^{cn} ),
\end{align*}
so that
\begin{align*}
 \P(Z_{\tau_n(N)}\geq N+M \ \big\vert Z_n\leq e^{cn}, \tau_n(N) =b, Z_ {\tau_n(N)-1} =a
 ) &\le \P_N(Z_{1}\geq N+M )).
\end{align*}
  This is uniform with respect to $a$ and $b$ so that summing over
them yields
$$ \forall n\in\N, \quad  \P(Z_{\tau_n(N)}\geq N+M \ \big\vert Z_n\leq e^{cn} )\leq \P_N(Z_1\geq M+N),$$
which completes the proof letting $M\rightarrow \infty$.
\end{proof}

$\newline$

We can now prove the second part of Theorem  \ref{T:trajectoire gauche} in the case $\P(\p(1)>0)>0$ (case a).
Let $\epsilon,\eta>0$  and $M,N\geq 1$ and note that
\bea
&& \P\big(\sup_{t\in [0,1]} \{\vert \log(Z_{[nt]})/n- f_c(t) \vert  \}\geq  \eta \ \vert \ Z_n\leq e^{cn} \big) \nonumber\\
&\leq & \underset{A_n}{ \underbrace{\P\big(\sup_{t\in [0,1]} \{\vert \log(Z_{[nt]})/n-
f_c(t) \vert  \}\geq  \eta, \
\tau_N/n \in [t_c-\epsilon,t_c+\epsilon], \ Z_{\tau (N)}\leq N+M
\ \vert \ Z_n\leq \exp(cn)\big)}} \nonumber \\
&&
\label{troispieces}
+ \underset{B_n}{\underbrace{\P\big( \tau (N)/n \not \in [t_c-\epsilon,t_c+\epsilon]\ \vert \ Z_n\leq \exp(cn)  \big)}}+ \underset{C_n}{\underbrace{\P\big( Z_{\tau (N)}\geq N+M \vert \ Z_n\leq \exp(cn)  \big)}}.
\eea
$\newline$

Thanks to Lemma \ref{P:cost with large pop} (ii), there exists $N$ large enough so that
$$B_n \stackrel{n\rightarrow \infty}{\longrightarrow}0 .$$

Then, by Lemma \ref{saut}, we can find
$M $ such that for $n$ large enough
$$C_n\leq\epsilon.$$

Finally, for every $\epsilon<\eta/2c$, for $n$ large enough,
$$
\sup_{t\in [0,t_c+\epsilon]} \{\vert \log(N)/n \vert + \vert f_c(t) \vert  \}\leq  \eta/2, $$
so that conditionally on the event $\{\tau_N/n \in [t_c-\epsilon,t_c+\epsilon]\}$,
$$
\sup_{t\in [0,\tau_N/n[} \{\vert \log(Z_{[nt]})/n- f_c(t) \vert  \}<\eta.
$$
Then, fixing $\epsilon>0$ such that
$$ \sup_{t_c-\epsilon\leq \alpha\leq t_c+\epsilon,\ t\in[0,1]} \{ f_c(\alpha+t)-ct/(1-\alpha)\}\leq \eta/2,$$
we have for every $n\in\N$,
\Bea
& A_n & \\
 &\leq &
 \P\big( \sup_{\tau_N/n \leq t\leq 1} \{ \vert \log(Z_{[nt]})/n - f_c(t)   \vert \}
 \geq \eta, \ \tau_N/n \in [t_c-\epsilon,t_c+\epsilon], \ Z_{\tau_N}\leq N+M \ \vert
 \ Z_n\leq \exp(cn)  \big) \\
 &\leq & \sup_{\substack{z_0 \in [N,N+M] \\ t_c-\epsilon\leq  \alpha \leq t_c+\epsilon }}
\P_{z_0}\big( \sup_{t \leq 1-\alpha} \{ \vert \log(Z_{[nt]})/n - f_c(\alpha+t)   \vert \}
 \geq \eta\ \vert \ Z_{[(1-\alpha) n]}\leq \exp(cn)   \big), \\ &\leq & \sup_{\substack{z_0 \in [N,N+M] \\ t_c-\epsilon\leq \alpha \leq t_c+\epsilon }}
\P_{z_0}\big( \sup_{t \leq 1-\alpha} \{ \vert \log(Z_{[nt]})/n - \frac{ct}{1-\alpha}    \vert \}
 \geq \eta/2\ \vert \ Z_{[n (1-\alpha )]}\leq \exp(cn) \big)\\
  &\leq & \sup_{\substack{z_0 \in [N,N+M] \\ c/(1-t_c+\epsilon)\leq  x \leq c/(1-t_c-\epsilon )}}
\P_{z_0}\big( \sup_{t \leq  c/x } \{ \vert \log(Z_{[nt]})/n - xt   \vert \}
 \geq \eta/2 \vert \ Z_{[nc/x]}\leq  \exp(nc/x.x) \big)
  \Eea
By ($\ref{limunif}$),  there exists $\epsilon>0$ such that $A_n \stackrel{n\rightarrow \infty}{\longrightarrow} 0.$ Then using ($\ref{troispieces}$),
$$\P\big(\sup_{t\in [0,1]} \{\vert \log(Z_{[nt]})/n- f_c(t) \vert  \}\geq  \eta \ \vert \ Z_n\leq e^{cn} \big)
\stackrel{n\rightarrow \infty}{\longrightarrow }0.$$
Thus in the case $\P(\p(1)>0)>0$, we  get that
conditionally on   $Z_n \leq e^{ c n }$,
$$\sup_{t\in[0,1]} \{ \big\vert \log(Z_{[tn]})/n- f_c(t) \big\vert  \} \ \stackrel{n\rightarrow \infty}{\longrightarrow }0, \qquad \t{in} \ \P.$$
The case $\P(\p(1)>0)=0$ is easier (and amounts to make $t_c=0$ in the proof above). \\

\section{Proof for upper deviation}
\label{S:proof upper}
Here, we assume that for every $k\geq 1$,
$$\E(Z_1^k)<\infty.$$

\begin{Lem}
\label{lemtechnq}
For every $c\geq \bar{L}$, denoting by
$$s_{max}:=\sup\{ s>1 : \E(m(\p)^{1-s})<1\},$$ we have for every $z_0\geq 1$,
$$\liminf_{n\rightarrow \infty} \inf_{z_0\geq 1} \{-\frac{1}{n}\log\big(\P_{z_0}(Z_n \geq z_0 \exp(cn))\big)
\}\geq \sup_{0\leq \eta\leq c-\bar{L}} \min(s_{max}\eta, \ \psi(c-\eta)).$$
\end{Lem}
The first part of Theorem \ref{upperth} is a direct consequence of this lemma. Indeed,  in
the case when $Z_n$ is strongly supercritical, $s_{max}=\infty$, then letting $\eta\downarrow 0$,
 we get, for every $c\geq  \bar{L}$,
$$- \log (\P_1(Z_n \leq e^{ c n }))/n  \stackrel{n\rightarrow \infty}{\longrightarrow } \psi(c).$$
$\newline$

\begin{proof}[Proof of Lemma \ref{lemtechnq}] For every $\eta>0$,
% and $\eta>0$ , define
%$$\mu(\epsilon,\eta):=(c+\epsilon)/ (c\exp(\eta)))$$
% > 1, \qquad \lim_{\epsilon\rightarrow 0+} \eta(\epsilon)=0.$$
$\P_{z_0}(Z_n \geq z_0 \exp(cn))$ is smaller than
\be
\label{maj1}
\P_{z_0}(Z_n \geq z_0 \exp(cn)), \ S_n\leq n[c-\eta] )+
\P_{z_0}(Z_n \geq z_0 \exp(cn)), \ S_n\geq n[c-\eta] ).
\ee

First, as for every $k\geq 1$, $\E(Z_1^k)<\infty$, by Theorem 3 in \cite{Guiv}, for every $s>1$ such that
$$\E(m(\p)^{1-s})<1,$$
there exists $C_s>0$ such that for every $n\in\N$,
$$\E_1(W_n^s)\leq C_s,$$
where  $W_n=\exp(-S_n)Z_n$. Note that conditionally on the environments $(\p_i)_{i=0}^{n-1}$,
$W_n$ starting from $z_0$ is the sum of $z_0$ iid random variable distributed as $W_n$ starting from $1$. Thus,
 there exists $C_s'$ such that for all $n,z_0\in\N$,
$$\E_{z_0}(W_n^s)\leq z_0^sC_s'.$$
Then,
\bea
\P_{z_0}(Z_n \geq z_0 \exp(cn), \ S_n\leq n[c-\eta] )
 &\leq & \P_{z_0}( Z_n\exp(-S_n) \geq z_0 \exp(n\eta)) \nonumber \\
&=& \P_{z_0}(W_n \geq z_0\exp(n\eta)) \nonumber \\
&\leq & \frac{\E_{z_0}(W_n^s)}{z_0^s\exp(ns\eta)} \nonumber \\
\label{maj2}
&\leq & C_s' \exp(-sn\eta).
\eea

Second, by (\ref{E:simple large dev bound upper}), we have
\be
\label{maj3}
\P_{z_0}(Z_n \geq \exp(cn), \ S_n\geq n[c-\eta] )\leq
\P(S_n\geq n[c-\eta] )\leq \exp(-n\psi(c-\eta)).
\ee

Combining (\ref{maj1}),(\ref{maj2}), and (\ref{maj3})  we get
$$\liminf_{n\rightarrow \infty} \inf_{z_0\geq 1}\{-\log\big(\P_{z_0}(Z_n\geq z_0\exp(cn) \big)\}/n \geq
\min(s\eta, \ \psi(c-\eta)).$$
Thus,
\Bea
\liminf_{n\rightarrow \infty} \inf_{z_0\geq 1}-\log\big(\P(\log(Z_n)/n \geq c)\big)\}/n
\geq  \sup_{0\leq \eta\leq c-\bar{L}} \min(s\eta, \ \psi(c-\eta)).
\Eea
Letting $s\uparrow  s_{max}=\sup\{ s>1 : \E(m(\p)^{1-s})<1\}$ yields the result.
\end{proof}
$\newline$

The proof of the second part of Theorem \ref{upperth} follows the proof of Proposition \ref{P:la ligne droite}.
Roughly speaking, for all $t\in (0,1)$ and $\epsilon>0$,
$$ \P(Z_{[nt]}= \exp(tn(c+\epsilon)), \ Z_n \geq \exp(cn))= \P(Z_{[nt]}= \exp(tn(c+\epsilon)))\P_{\exp(tn(c+\epsilon))}
(Z_{n-[nt]}\geq \exp(cn)).$$
Then the first part of Theorem \ref{upperth} ensures that
\Bea
&&\lim_{n\rightarrow \infty}-\log(\P(Z_{[nt]}= \exp(tn(c+\epsilon)), \ Z_n\geq \exp(cn)))/n\\
&& \qquad \qquad = t\psi(c+\epsilon)+(1-t)\psi(c-t/(1-t)\epsilon) \\
&& \qquad \qquad > \psi(c),
\Eea
by strict convexity of $\psi$. This entails that $\log(Z_{[nt]})/n\rightarrow ct$ as $n\rightarrow \infty$.

\section{Proof without supercritical environments}
\label{S:subcritical}
We assume here that $\P(m(\p)\leq 1)=1$.
Recall that $f_i$ is the probability generating function of $\p_i$ and that, denoting by
$$F_n:=f_{0}\circ \cdots\circ f_{n-1},$$
 we have for every $k\in\N$,
$$\E_k(s^{Z_{n+1}}\ \vert \ \ f_0, ..., \ f_n)=F_{n+1}(s)^k \qquad (0\leq s\leq 1).$$
We assume also that for every $j\geq 1$, there exists $M_j>0$ such that
$$\sum_{k=0}^{\infty} k^j \p(k) \leq M_j \quad \t{a.s.}$$
Then,
$$ f^{(j)}(1)\leq M_j \quad \t{a.s.}$$

We use that for ever $c>1$ and $k\geq 1$, by Markov inequality,
\Bea
\P(Z_n\geq c^n)&=& \P(Z_n(Z_{n}-1)...(Z_{n}-k+1) \geq c^n(c^n-1)...(c^n-k+1)) \\
&\leq &\frac{\E(Z_n(Z_{n}-1)...(Z_{n}-k+1))}{c^n(c^n-1)...(c^n-k+1)}\\
%&\leq &\frac{\E(Z_nZ_{n-1}....Z_{n-k+1})}{c^n(c^n-1)...(c^n-k+1)}\\
&=& \frac{\E(F_n^{(k)}(1))}{c^n(c^n-1)...(c^n-k+1)}.
\Eea
Thus, to get Proposition \ref{totalrecall}, it is enough to prove that
for every $k>1$,
$$\E(F_n^{(k)}(1))\leq C_kn^{k^k}$$
and let $k\rightarrow \infty$. The last inequality can be directly derived
from the following lemma, since here $f_i'(1)\leq 1$ a.s. and there exists $M_j>0$ such that for every $j\in \N$, $f^{(j)}(1)\leq M_j$ a.s.
\begin{Lem}
Let $(g_i)_{1\leq i\leq n}$ be power series with positive coefficients such that
$$\forall 2\leq i\leq n, \quad g_i(1)=1$$
and denote by
$$G_i=g_{i}\circ ...\circ g_n, \qquad (1\leq i\leq n). $$
Then, for every $k\geq 0$,
$$
\sup_{x\in[0,1]}
G_1^{(k)} (x) \leq  \max_{\substack{0\leq j\leq k \\ 1\leq i\leq n }}
(1,[g_i^{(j)}(1)]^{k^k}).\max_{2\leq i \leq n}(1,g'_i(1))^{nk} .n^{k^k}$$
\end{Lem}
\begin{proof}
This result can be proved by induction. Indeed,

\Bea
G_1^{(k+1)} &=&[\Pi_{i=1}^n g_i'\circ G_{i+1}]^{(k)} \\
&=&\sum_{k_1+...+k_n=k} \Pi_{i=1}^n [g_i'\circ G_{i+1}]^{(k_i)}.
\Eea
Then, noting that $\#\{i\in [1,n] : k_i>0\}\leq k$ and $\#\{k_i : k_1+...+k_n=k\}\leq n^k$, for every $x\in [0,1]$,
$$ G_1^{(k+1)}(x) \leq n^k \max_{\substack{1\leq i\leq n \\ 0\leq k_i\leq k }} \{1,[g_i'\circ G_{i+1}]^{(k_i)} (x) \}^k.\max(1,g'_1(G_{2}(x))).\max_{2\leq i \leq n}(1,g'_i(1))^{n} .$$
So,
$$\sup_{x\in[0,1]} G_1^{(k+1)}(x) \leq n^k \max_{\substack{1\leq i\leq n \\ 0\leq k_i\leq k }} \{1,[g_i'\circ G_{i+1}]^{(k_i)} (x) \}^{k+1}.\max_{2\leq i \leq n}(1,g'_i(1))^{n} .$$
One can complete the induction noting  that $k+k^k(k+1)\leq (k+1)^{k+1}$.
\end{proof}

\medskip

\noindent {\bf Acknowledgements:} The authors wish to thank Amaury
Lambert for many interesting discussions at the start of this
project.


\begin{thebibliography}{10}
\bibitem{afa} V. I. Afanasyev, J. Geiger, G. Kersting, V.A. Vatutin (2005). Criticality for branching processes in random environment. \emph{Ann. Probab. }33. 645-673.
\bibitem{at} K. B. Athreya, S. Karlin (1971).  On branching processes with random environments, I : Extinction probability. \emph{Ann. Math. Stat. }42. 1499-1520.
\bibitem{atr} K. B. Athreya, S. Karlin (1971).  On branching processes with random environments, II : limit theorems. \emph{Ann. Math. Stat. }42. 1843-1858.
\bibitem{atld2} K. B. Athreya; A. Vidyashankar (1993). Large deviation results for branching processes. \emph{Stochastic processes}, 7-12, Springer, New York.
\bibitem{atld} K. B. Athreya (1994). Large deviation rates for branching processes. I  Single type case. \emph{Ann. Appl. Probab. } 4, no. 3, 779-790.
\bibitem{AN}  K. B. Athreya, P. E. Ney (2004). \emph{Branching processes}. Dover Publications, Inc., Mineola, NY.
%\bibitem{vb} V. Bansaye. Proliferating parasites in dividing cells : Kimmel's branching revisited.
\bibitem{vb} V. Bansaye  (2008). Proliferating parasites in dividing cells : Kimmel's branching model revisited. \emph{Ann. Appl. Prob.} Vol. 18, No. 3, 967-996.
\bibitem{BB} J. D. Biggins, N. H. Bingham (1993). Large deviations in the supercritical branching process. \emph{Adv. in Appl. Probab.}  25, no. 4, 757-772.
\bibitem{Bing} N. H. Bingham (1988).  On the limit of a supercritical branching process. \emph{J. Appl. Prob.} 25A, 215-228.
\bibitem{Dek} F.M. Dekking (1988). On the survival probability of  a branching process in a finite state iid environment. \emph{Stochastic Processes and their Applications }27 p151-157.
\bibitem{dembo zeitouni} A. Dembo, O. Zeitouni (1998).  \emph{Large deviations techniques and applications}. Second edition. Applications of Mathematics (New York), 38. Springer-Verlag, New York.
\bibitem{Dubuc} S. Dubuc (1971). La densite de la loi limite d'un processus
en casacade expansif. \emph{Z. Wahr.}  19, 281-290.
\bibitem{Dubuc2} S. Dubuc (1971). Problemes relatifs a l'iteration des fonctions
suggeres par les processus en cascade. \emph{Ann. Inst. Fourier} 21, 171-251.
\bibitem{Fleischmann} K. Fleischmann, V. Wachtel (2007). Lower deviation probabilities for supercritical Galton-Watson processes.  \emph{Ann. I. H. Poincaré Probab. Statist}.  43 ,  no. 2, 233--255
\bibitem{bpree} J. Geiger, G. Kersting, V A. Vatutin (2003). Limit Theorems for subcritical branching processes in random environment. \emph{Ann. I. H. Poincaré }39, no. 4.  593-620.
\bibitem{Guiv} Y. Guivarc'h, Q. Liu (2001). Asymptotic properties of branching processes in random environment. \emph{C.R. Acad. Sci. Paris,} t.332, Serie I. 339-344.
\bibitem{Hambly} B. Hambly (1992). On the limiting distribution of a supercritical branching process in random environment. \emph{J. Appl. Probab. } 29, No 3, 499-518.
\bibitem{kozld} M. V. Kozlov (2006).  On large deviations of branching processes in a random environment: a geometric distribution of the number of descendants. \emph{Diskret. Mat. 18} , no. 2, 29--47.
\bibitem{Liu}  Q. Liu (2000). On generalized multiplicative cascades. \emph{Stochastic Process. Appl.} 86, no. 2, 263-286.
\bibitem{Morters} P. M\"{o}rters, M. Ortgiese  (2008). Small value probabilities via the branching tree heuristic.
\emph{Bernoulli} 14, no. 1, 277--299.
\bibitem{Ney} P. E. Ney; A.  N. Vidyashankar (2004). Local limit theory and large deviations for supercritical branching processes. \emph{Ann. Appl. Probab.} 14, no. 3, 1135--1166
\bibitem{Rouault} A. Rouault (2000). Large deviations and branching processes. \emph{Proceedings of the 9th International Summer School on Probability Theory and Mathematical Statistics} (Sozopol, 1997). Pliska Stud. Math. Bulgar. 13, 15--38.
\bibitem{bpre} W.L. Smith, W. Wilkinson (1969). On branching processes in random environments. \emph{Ann. Math. Stat.} Vol40, No 3, p814-827.
\end{thebibliography}
\end{document}